\documentclass[11pt, oneside]{amsart}

\usepackage[british,english]{babel} 
\usepackage{graphics,color,pgf}
\usepackage{epsfig}
\usepackage[ansinew]{inputenc}
\usepackage[all]{xy}
\usepackage{hyperref}
\newdir{ >}{!/8pt/@{}*@{>}}
\usepackage{amssymb, amsmath,amsthm, mathtools, amscd}
\usepackage{mathrsfs}
\usepackage{stmaryrd}
\usepackage[margin=1.5in]{geometry}
\usepackage{enumerate}

\makeatletter
\newtheorem*{rep@theorem}{\rep@title}
\newcommand{\newreptheorem}[2]{%
\newenvironment{rep#1}[1]{%
 \def\rep@title{#2 \ref{##1}}%
 \begin{rep@theorem}}%
 {\end{rep@theorem}}}
\makeatother

\newtheorem{intro_thm}{Theorem}

\newtheorem{intro_prop}[intro_thm]{Proposition}

\theoremstyle{plain}
\newtheorem{teor}{Theorem}[section]
\newreptheorem{teor}{Theorem}  
\newtheorem{lem}[teor]{Lemma}
\newtheorem{cor}[teor]{Corollary}
\newtheorem{prop}[teor]{Proposition}
\newreptheorem{prop}{Proposition}  

\newtheorem{setup}[teor]{Setup}
\newtheorem{assumption}[teor]{Assumption}

\theoremstyle{definition}
\newtheorem{deft}[teor]{Definition}

\theoremstyle{remark}
\newtheorem{oss}[teor]{Remark}

\DeclareMathOperator\rk{rk}

\DeclareMathOperator\upC{\textup{C}}

\DeclareMathOperator\upH{\textup{H}}

\DeclareMathOperator\upJ{\textup{J}}

\DeclareMathOperator\upL{\textup{L}}

\DeclareMathOperator\upT{\textup{T}}

\DeclareMathOperator\bbC{\mathbb{C}}

\DeclareMathOperator\bbH{\mathbb{H}}

\DeclareMathOperator\bbN{\mathbb{N}}

\DeclareMathOperator\bbR{\mathbb{R}}
\DeclareMathOperator\bbS{\mathbb{S}}

\DeclareMathOperator\bbZ{\mathbb{Z}}

\DeclareMathOperator\calB{\mathcal{B}}

\DeclareMathOperator\calD{\mathcal{D}}

\DeclareMathOperator\calJ{\mathcal{J}}

\DeclareMathOperator\calM{\mathcal{M}}

\DeclareMathOperator\calO{\mathcal{O}}

\DeclareMathOperator\calS{\mathcal{S}}
\DeclareMathOperator\calT{\mathcal{T}}

\DeclareMathOperator\calX{\mathcal{X}}
\DeclareMathOperator\calY{\mathcal{Y}}

\DeclareMathOperator\pu{\textup{PU}}

\begin{document}

\title[Maximal measurable cocycles of surface groups]{Algebraic hull of maximal measurable cocycles of surface groups into Hermitian Lie groups}

\author[]{A. Savini}
\address{Department of Mathematics, University of Bologna, Piazza di Porta San Donato 5, 40126 Bologna, Italy}
\email{alessio.savini5@unibo.it}

\thanks{}

\keywords{Hermitian Lie group, tube type, tightness, Shilov boundary, maximal measurable cocycle, K\"ahler form, Toledo invariant}
\subjclass[2010]{}
\date{\today.\ \copyright{\ A. Savini 2020}
The author was partially supported by the project \emph{Geometric and harmonic analysis with applications}, funded by EU Horizon 2020 under the Marie Curie grant agreement No 777822.}.
\maketitle

\begin{abstract}
Following the work of Burger, Iozzi and Wienhard for representations, in this paper we introduce the notion of maximal measurable cocycles of a surface group. More precisely, let $\mathbf{G}$ be a semisimple algebraic $\bbR$-group such that $G=\mathbf{G}(\bbR)^\circ$ is of Hermitian type. If $\Gamma \leq L$ is a torsion-free lattice of a finite connected covering of $\pu(1,1)$, given a standard Borel probability $\Gamma$-space $(\Omega,\mu_\Omega)$, we introduce the notion of Toledo invariant for a measurable cocycle $\sigma:\Gamma \times \Omega \rightarrow G$.

The Toledo invariant remains unchanged along $G$-cohomology classes and its absolute value is bounded by the rank of $G$. This allows to define maximal measurable cocycles. We show that the algebraic hull $\mathbf{H}$ of a maximal cocycle $\sigma$ is reductive and the centralizer of $H=\mathbf{H}(\bbR)^\circ$ is compact. If additionally $\sigma$ admits a boundary map, then $H$ is of tube type and $\sigma$ is cohomologous to a cocycle stabilizing a unique maximal tube type subdomain. This result is analogous to the one obtained for representations. 

In the particular case $G=\textup{PU}(n,1)$ maximality is sufficient to prove that $\sigma$ is cohomologous to a cocycle preserving a complex geodesic. 

We conclude with some remarks about boundary maps of maximal Zariski dense cocycles. 
\end{abstract}

\section{Introduction}

Given a torsion-free lattice $\Gamma \leq L$ in a semisimple Lie group $L$, any representation $\rho:\Gamma \rightarrow H$ into a locally compact group $H$ induces a well-defined map at the level of continuous bounded cohomology groups. Hence fixed a preferred bounded class in the cohomology of $H$, one can pullback it and compare the resulting class with the fundamental class determined by $\Gamma$ via Kronecker pairing. This is a standard way to obtain \emph{numerical invariants} for representations, whose importance has become evident in the study of rigidity and superrigidity properties. In many cases (such as the Toledo invariant, the Volume invariant or the Borel invariant) a numerical invariant has bounded absolute value and the maximum is attained if and only if the representation can be extended to a representation $L \rightarrow H$ of the ambient group.  

Several examples of these phenomena are given by the work of Bucher, Burger, Iozzi \cite{iozzi02:articolo,bucher2:articolo,BBIborel} in the case of representations of real hyperbolic lattices, by Burger and Iozzi \cite{BIcartan} and by Duchesne and Pozzetti \cite{Pozzetti,duchesne:pozzetti} for complex hyperbolic lattices and by the work of Burger, Iozzi and Wienhard \cite{BIW07,BIW09,BIW1} when the target group is of Hermitian type. In the latter case, of remarkable interest is the analysis of the representation space $\textup{Hom}(\Gamma,G)$ when $G$ is a group of Hermitian type and $\Gamma$ is a lattice in a finite connected covering of  $\pu(1,1)$, that is a hyperbolic surface group. Burger, Iozzi and Wienhard \cite{BIW1} 
exploited the existence of a natural K\"ahler structure on the Hermitian symmetric space associated to $G$ in order to define the notion of \emph{Toledo invariant} of a representation $\rho:\Gamma \rightarrow G$. That invariant has bounded absolute value and its maximality has important consequences on the Zariski closure $\mathbf{H}=\overline{\rho(\Gamma)}^Z$ of the image of the representation. Indeed the authors show that in the case of maximality $\mathbf{H}$ is reductive, $H=\mathbf{H}(\bbR)^\circ$ has compact centralizer and it is of tube type and the representation $\rho$ is injective with discrete image and it preserves a unique maximal tube type subdomain \cite[Theorem 5]{BIW1}. A domain is of \emph{tube type} if it can be written in the form $V+i\Omega$, where $V$ is a real vector space and $\Omega \subset V$ is an open convex cone. Maximal tube type subdomains in a Hermitian symmetric space $\calX$ generalize the notion of complex geodesic in $\bbH^n_{\bbC}$ and they are all $G$-conjugated. 

A source of inspiration for \cite[Theorem 5]{BIW1} is represented by the work of Toledo \cite{Toledo89}. Indeed he proved that maximal representations into $\textup{PU}(n,1)$ must preserve a complex geodesic. Partial results in the direction of \cite[Theorem 5]{BIW1} were obtained by several authors. It is worth mentioning the papers by Hern\'andez \cite{Her91}, by Koziarz and Maubon \cite{koziarz:maubon} and by Bradlow, Garc\'ia-Prada and Gothen \cite{garcia:geom,garcia:dedicata}. In the latter case those results were obtained using different techniques based on the notion of Higgs bundle (see also \cite{koziarz:maubon:2} for the study of representations of complex hyperbolic lattices via Higgs bundles).

It is worth noticing that in the particular case of split real groups and surfaces without boundary, the set of maximal representations contains the Hitchin component \cite{hitchin}. The Hitchin component has been sistematically studied by serveral mathematicians. For instance Labourie \cite{labourie} focused his attention on the Asonov property, whereas Fock and Goncharov \cite{fock:hautes,Fock:adv} related the Hitchin component with the notion of Lusztig's positivity. 

A crucial point in the proof of \cite[Theorem 5]{BIW1} is that maximal representations are \emph{tight}, that is the seminorm of the pullback of the bounded  K\"ahler class $\kappa^b_G$ is equal to the norm of $\kappa^b_G$. The tightness property has an analytic counterpart in terms of maps between symmetric spaces and Burger, Iozzi and Wienhard \cite{BIW09} give a complete characterization of tight subgroups of a Lie group of Hermitian type.

Recently the author \cite{savini3:articolo} together with Moraschini \cite{moraschini:savini,moraschini:savini:2} and Sarti \cite{savini:sarti} has applied bounded cohomology techniques to the study measurable cocycles with an essentially unique boundary map. The existence of a boundary map allows to define a pullback in bounded cohomology as in \cite{burger:articolo} and hence to develop a theory of numerical invariants, called \emph{multiplicative constants}, also in the context of measurable cocycles. 

The main goal of this paper is the study of measurable cocycles of surface groups. Let $\Gamma \leq L$ be a torsion-free lattice of a finite connected covering $L$ of $\pu(1,1)$. Consider a standard Borel probability $\Gamma$-space $(\Omega,\mu_\Omega)$ and let $\mathbf{G}$ be a semisimple real algebraic group such that $G=\mathbf{G}(\bbR)^\circ$ is of Hermitian type. Using a measurable cocycle $\sigma:\Gamma \times \Omega \rightarrow G$, we can define a pullback map 
in bounded cohomology and hence mimic the techniques used in \cite{moraschini:savini,moraschini:savini:2} to define the \emph{Toledo invariant of $\sigma$}. In the particular case that $\sigma$ admits a boundary map $\phi:\bbS^1 \times \Omega \rightarrow \check{\calS}_G$ into the Shilov boundary of $G$, then we recover the same approach developed in \cite{moraschini:savini,moraschini:savini:2} (see Lemma \ref{lem:pullback:cocycle:boundary}). In an analogous way to what happens for representations, the Toledo invariant is constant along $G$-cohomology classes and has absolute value bounded by $\rk(\calX)$, the rank of the symmetric space $\calX$ associated to $G$. Thus it makes sense to speak about \emph{maximal measurable cocycles}. This will be a particular example of \emph{tight cocycles} (see Definition \ref{def:tight:cocycle}). 

Maximality allows to give a characterization of the \emph{algebraic hull} of a measurable cocycle, as stated in the following 

\begin{intro_thm}\label{teor:maximal:alghull}
Let $\Gamma \leq L$ be a torsion-free lattice of a finite connected covering $L$ of $\pu(1,1)$ and let $(\Omega,\mu_\Omega)$ be a standard Borel probability $\Gamma$-space. Let $\mathbf{G}$ be a semisimple algebraic $\bbR$-group such that $G=\mathbf{G}(\bbR)^\circ$ is a Lie group of Hermitian type. Consider a measurable cocycle $\sigma:\Gamma \times \Omega \rightarrow G$. Denote by $\mathbf{H}$ the algebraic hull of $\sigma$ in $\mathbf{G}$ and set $H=\mathbf{H}(\bbR)^\circ$. If $\sigma$ is maximal, then 
\begin{enumerate}
	\item the algebraic hull $\mathbf{H}$ is reductive;
	\item the centralizer $Z_G(H)$ is compact;
\end{enumerate}
If additionally $\sigma$ admits a boundary map $\phi:\bbS^1 \times \Omega \rightarrow \check{\calS}_{\calY}$ into the Shilov boundary of the symmetric space $\calY$ associated to $H$, then
\begin{enumerate}
	\item[(3)] the symmetric space $\calY$ is Hermitian of tube type;
	\item[(4)] it holds $\mathbf{H}(\bbR) \subset \textup{Isom}(\calT)$ for some maximal tube type subdomain $\calT$ of $\calX$. Equivalently there exists a cocycle cohomologous to $\sigma$ which preserves $\calT$. 
\end{enumerate}
\end{intro_thm}

The above theorem should be interpreted as a suitable adaptation of \cite[Theorem 5]{BIW1} to the context of maximal measurable cocycles. The first two properties are immediate consequences of the tightness of maximal cocycles, as shown in Theorem \ref{teor:alg:hull:tight}. The tube type condition is more involving and it is proved in Theorem \ref{teor:symmetric:tube}.  

It is worth mentioning that Theorem \ref{teor:maximal:alghull} can be translated in the language of principal $G$-bundles. Indeed suppose that $\Omega$ is actually a smooth manifold and $\sigma$ is the cocycle associated to the section of a principal $G$-bundle $P$ with a $\Gamma$-action. The maximality assumption is telling us that we can find a $\Gamma$-invariant principal subbundle of $P$ with fibers of tube type. This has application for instance to the reducibility of $G$-structures on $\Omega$. 

In the particular case $G=\textup{PU}(n,1)$, imitating \cite[Theorem 8]{BIgeo} and \cite[Theorem C]{koziarz:maubon}, maximality implies the existence of a cohomologous cocycle preserving a complex geodesic. Indeed we are going to prove the following

\begin{intro_prop}\label{prop:complex:geodesic}
Let $L$ a finite connected covering of $\textup{PU}(1,1)$ and let $\Gamma \leq L$ be a torsion-free lattice. Let $(\Omega,\mu_\Omega)$ be a standard Borel probability. If a measurable cocycle $\sigma:\Gamma \times \Omega \rightarrow \textup{PU}(n,1)$ is maximal, then it is cohomologous to a measurable cocycle which preserves a complex geodesic. 
\end{intro_prop}

Notice that in the previous statement the requirement of a boundary map is not necessary, since the existence of such a map will be part of the proof. 

We conclude with some remarks about boundary maps of maximal Zariski dense cocycles. For representations, the relation between maximality and boundary maps preserving positivity of triples were studied by Guichard \cite{Guichard}, Labourie \cite{labourie} and Fock and Goncharov \cite{fock:hautes}

Here we attempt to extend \cite[Theorem 5.2]{BIW1} to the context of measurable cocycles. Given a maximal Zariski dense cocycle, we can construct a boundary map which has left-continuous (respectively right-continuous) slices. Moreover each slice preserves \emph{transversality} and it is \emph{monotone}, as proved in Theorem \ref{teor:boundary:map}. Unfortunately, to get the statement, we need to make an additional assumption on the measurable map $\phi:\bbS^1 \times \Omega \rightarrow \check{\calS}_{\calX}$. More precisely we need to assume that the essential image of almost every slice intersects nicely all closed algebraic subset of $\check{\calS}_{\calX}$ (Assumption \ref{ass:zariski:zero:measure}). This assumption is clearly verified by cocycles cohomologous to maximal Zariski dense representations \cite[Proposition 5.2]{BIW1}, but we do not know more generally under which conditions on both $\sigma$ or $\phi$ this is true and it would be interesting to know it.  
The proof of Theorem \ref{teor:boundary:map} follows the line of \cite[Section 8]{BIWL} and of \cite[Theorem 5.2]{BIW1}.

\subsection*{Acknowledgements}

I would like to thank Michelle Bucher and Maria Beatrice Pozzetti for the useful conversations about continuous bounded cohomology. I am grateful to Maria Beatrice for having suggested me to define pullback maps induced by measurable cocycles without the use of boundary maps. 

I finally thank the referee for the careful work and the precious suggestions which allowed to improve the quality of the paper.

\subsection*{Plan of the paper}
In Section \ref{sec:preliminary} we recall the preliminary definitions and results that we need in the paper. In Section \ref{sec:measurable:cocycles} we remind the notion of measurable cocycle and of cohomology class determined by a cocycle. Of particular importance for our purposes will be the definitions of algebraic hull and Zariski density. Then we conclude the section with some elements of boundary theory. Section \ref{sec:burger:monod} is devoted to continuous and continuous bounded cohomology. We remind the functorial approach by Burger and Monod. In Section \ref{sec:no:boundary:map} we describe the theoretical background to define pullback in terms of measurable cocycles. When a boundary map exists we recover the approach already studied by the author and Moraschini (see Lemma \ref{lem:pullback:cocycle:boundary}). The last part is devoted to Hermitian symmetric spaces (Section \ref{sec:hermitian:groups}). 

The main theorem of paper is proved in Section \ref{sec:maximal:cocycles}. We first introduce the notion of Toledo invariant of a measurable cocycle in Section \ref{sec:toledo:invariant}. In Section \ref{sec:maximal:cocycle:thm} it appears the definition of maximal cocycle. Maximal cocycles are tight by Proposition \ref{prop:maximal:tight}. In Section \ref{sec:zariski:maximal} we focus our attention on maximal Zariski dense cocycles. The tightness property together with Theorem \ref{teor:symmetric:tube} allows to prove Theorem \ref{teor:maximal:alghull}. We conclude with Section \ref{sec:boundary:map}, where we prove Theorem \ref{teor:boundary:map}.

\section{Preliminary definitions and results}\label{sec:preliminary}

\subsection{Measurable cocycles}\label{sec:measurable:cocycles}

The following section is devoted to a quick review about measurable cocycles theory. We are going to recall the definitions of  both measurable cocycle and cohomology class. Then we will introduce the  notion of algebraic hull and we will conclude the section with some elements of boundary theory. For a more detailed discussion about those topics we refer the reader to the work of both Furstenberg \cite{furst:articolo73,furst:articolo} and Zimmer \cite{zimmer:preprint,zimmer:annals,zimmer:libro}. 

Consider two locally compact second countable groups $G,H$ endowed with their Haar measurable structure. Given a standard Borel measure space $(\Omega,\mu_\Omega)$ we say that it is a \emph{$G$-space} if $G$ acts on $\Omega$ by measure-preserving transformations. Additionally if $\mu_\Omega$ is a probability measure, we are going to call $(\Omega,\mu_\Omega)$ a \emph{standard Borel probability $G$-space}. Given another measure space $(\Theta,\mu_\Theta)$, we are going to denote by $\textup{Meas}(\Omega,\Theta)$ the space of measurable functions with the topology of the \emph{convergence in measure}. The latter is generated by the base of open sets of the form 
\begin{equation}\label{eq:open:set}
\mathscr{S}(K,h,\alpha,\varepsilon)=\{ f \in \textup{Meas}(\Omega,\Theta)| \  \mu_\Omega\{ x \in K | d(f(x),h(x))>\alpha \}<\varepsilon \} \ ,
\end{equation}
where $K \subseteq \Omega$ is measurable, $h \in \textup{Meas}(\Omega,\Theta)$ and $\alpha,\varepsilon >0$. Later on we are going to consider the Borel structure generated by the topology of convergence in measure, that is the measurable structure generated by countable unions and intersections of sets given by Equation \eqref{eq:open:set}. We refer the reader to \cite[Chapter VII.1]{margulis:libro} for more details. 

\begin{deft}\label{def:measurable:cocycle}
Let $G,H$ two locally compact second countable groups and let $(\Omega,\mu_\Omega)$ be a standard Borel probability $G$-space. A measurable function $\sigma:G \times \Omega \rightarrow H$ is a \emph{measurable cocycle} if it holds
\begin{equation}\label{eq:measurable:cocycle}
\sigma(g_1 g_2,s)=\sigma(g_1,g_2 s)\sigma(g_2,s) \ ,
\end{equation}
for almost every $g_1,g_2 \in G$ and almost every $s \in \Omega$. 
\end{deft}

Measurable cocycles are quite ubiquitous in Mathematics and Equation \eqref{eq:measurable:cocycle} can be suitably interpreted as a naive generalization to the measurable context of the chain rule for differentiation of smooth functions. By writing a measurable cocycle $\sigma$ as an element $\sigma \in \textup{Meas}(G,\textup{Meas}(\Omega,H))$, Equation \eqref{eq:measurable:cocycle} boils down the cocycle condition. Indeed $\sigma$ may be interpreted as a Borel $1$-cocycle in the sense of Eilenberg-MacLane (see \cite{feldman:moore,zimmer:preprint} for more details about this interpretation). Following this line, one could naturally ask when two different cocycles are cohomologous.

\begin{deft}\label{def:cohomologous:cocycles}
Let $\sigma:G \times \Omega \rightarrow H$ be a measurable cocycle and let $f:\Omega \rightarrow H$ be a measurable function. The \emph{$f$-twisted cocycle of $\sigma$} is defined as 
$$
\sigma^f:G \times \Omega \rightarrow H, \ \ \sigma^f(g,s):=f(gs)^{-1}\sigma(g,s)f(s) \ . 
$$
We say that two cocycles $\sigma_1,\sigma_2:G \times \Omega \rightarrow H$ are \emph{cohomologous} if there exists a measurable function $f:\Omega \rightarrow H$ such that 
$$
\sigma_2^f=\sigma_1 \ . 
$$
\end{deft}

Choosing a measurable function $f:\Omega \rightarrow H$ is a typical way to construct cocycles starting from representations. Indeed, given a continuous representation $\rho:G \rightarrow H$, one can verifiy that the measurable function 
$$
\sigma_\rho:G \times \Omega \rightarrow H \ , \ \ \sigma_\rho(g,s):=\rho(g) \ , 
$$
is a measurable cocycle as a consequence of the morphism condition. This allows to see representations theory into the wider world of measurable cocycles theory. Additionally this offers us the possibility to interpret the notion of cohomologous cocycles as a generalization of conjugated representations. 

Given a representation $\rho:G \rightarrow H$, if the image is not closed, it is quite natural to consider its closure, which it is still a subgroup of $H$. Unfortunately the image of a cocycle has no structure a priori. Nevertheless, if $H$ corresponds to the real points of a real algebraic group, then there is a notion which is in some sense similar to take the closure of the image of a representation.
\begin{deft}
Suppose that $\mathbf{H}$ is a real algebraic group.  Let $\sigma:G \times \Omega \rightarrow \mathbf{H}(\bbR)$ be a measurable cocycle. The \emph{algebraic hull associated to $\sigma$} is (the conjugacy class of) the smallest algebraic subgroup $\mathbf{L}$ of $\mathbf{H}$ such that $\mathbf{L}(\bbR)$ contains the image of a cocycle cohomologous to $\sigma$. 
\end{deft}

As proved in \cite[Proposition 9.2]{zimmer:libro} this notion is well-defined by the descending chain condition on algebraic subgroups and it depends only the cohomology class of the cocycle. It is worth noticing that the algebraic hull can be exploited to give a concept of Zariski density for measurable cocycles. More precisely, we have the following
\begin{deft}\label{def:zariski:dense}
Consider $\mathbf{H}$ a real algebraic group. We say that a measurable cocycle $\sigma:G \times \Omega \rightarrow \mathbf{H}(\mathbb{R})$ is \emph{Zariski dense} it its algebraic hull is exactly $\mathbf{H}$. 
\end{deft}
In Section \ref{sec:zariski:maximal} we are going to focus our attention on Zariski dense cocycles of surface groups to obtain important properties for their targets.

We conclude this brief discussion about measurable cocycles introducing some elements of boundary theory. In order to do this, we are going to assume that $G$ is a semisimple Lie group of non-compact type. Let $P$ be a \emph{minimal parabolic subgroup} of $G$ and suppose that $H$ acts measurably on a measure space $(Y,\nu)$ by preserving the measure class of $\nu$. 

\begin{deft}\label{def:boundary:map}
Let $\sigma:G \times \Omega \rightarrow H$ be a measurable cocycle. A \emph{(generalized) boundary map} is a measurable map $\phi:G/P \times \Omega \rightarrow Y$ which is $\sigma$-equivariant, that is 
$$
\phi(g \xi,g s)=\sigma(g,s)\phi(\xi,s) \ , 
$$
for every $g \in G$ and almost every $\xi \in G/P, s \in \Omega$. 
\end{deft}

It is easy to check that, if $\phi:G/P \times \Omega \rightarrow Y$ is a boundary map for $\sigma$, then $\phi^f:G/P \times \Omega \rightarrow Y, \ \phi^f(\xi,s):=f(s)^{-1}\phi(\xi,s)$ is a boundary map for $\sigma^f$ for any measurable function $f:\Omega \rightarrow H$. 

The existence and the uniqueness of a boundary map associated to a cocycle $\sigma$ rely on the dynamical properties of $\sigma$. For a more detailed discussion about it we refer the reader to \cite{furst:articolo}. Boundary maps for measurable cocycles will be crucial in Section \ref{sec:zariski:maximal} to study the propreties of the target group of maximal measurable cocycles of surface groups.  

\subsection{Continuous bounded cohomology and functorial approach}\label{sec:burger:monod}

Given a locally compact group $G$ we are going to remind the notion of continuous and continuous bounded cohomology groups of $G$. A remarkable aspect of continuous bounded cohomology is that it can be computed using any strong resolution by relatively injective modules. For more details about continuous bounded cohomology and its functorial approach we refer to the work of Burger and Monod \cite{burger2:articolo,monod:libro}.

 We consider the set of \emph{real continuous bounded functions on $G^{\bullet+1}$} given by
\begin{align*}
\upC^\bullet_{cb}(G;\mathbb{R}):=\{ f : G^{\bullet+1} \rightarrow \mathbb{R} \ | & \ f \ \textup{is continuous and} \\
&\lVert f \rVert_\infty:=\sup_{g_0,\ldots,g_{\bullet}} | f(g_0,\ldots,g_{\bullet}) | < \infty \} \ , 
\end{align*}
where $| \ \cdot \ |$ is the usual absolute value on $\mathbb{R}$. Each $\upC^\bullet_{cb}(G;\mathbb{R})$ is a normed via the supremum norm and it can be endowed with an isometric action of $G$ defined by
\begin{equation}\label{eq:left:action}
(gf)(g_0,\ldots,g_\bullet):=f(g^{-1}g_0,\ldots,g^{-1}g_\bullet) \ ,
\end{equation}
where $f \in \upC^\bullet_{cb}(G;\bbR)$ and $g,g_0,\ldots,g_\bullet \in G$.  Notice that in this case $\bbR$ is endowed with the structure of \emph{trivial $G$-module}. Defining the \emph{standard homogeneous coboundary operator} by
$$
\delta^\bullet:\upC^\bullet_{cb}(G;\bbR) \rightarrow \upC^{\bullet+1}_{cb}(G;\bbR) \ , 
$$
$$
\delta^\bullet(f)(g_0,\ldots,g_{\bullet+1}):=\sum_{i=0}^{\bullet+1}(-1)^i f(g_0,\ldots,\hat g_i,\ldots,g_{\bullet+1}) \ ,
$$
we get a cochain complex $(\upC^\bullet_{cb}(G;\bbR),\delta^\bullet)$.  

\begin{deft}\label{def:bounded:cohomology}
Let $G$ be a locally compact group. The \emph{ $k$-th continuous bounded cohomology group} of $G$ is the $k$-th cohomology group of the $G$-invariant subcomplex $(\upC^\bullet_{cb}(G;\bbR)^G,\delta^\bullet)$, that is 
$$
\upH^k_{cb}(G;\bbR):=\upH^k(\upC^\bullet_{cb}(G;\bbR)^G) \ ,
$$
for every $k \geq 0$. 
\end{deft}

It is worth noticing that each cohomology group $\upH^\bullet_{cb}(G;\bbR)$ has a natural seminormed structure inherited by the normed structure on the continuous bounded cochains. 

By dropping the assumption of boundedness one can define similarly the complex of continuous cochains $(\upC^\bullet_c(G;\bbR),\delta^\bullet)$ and the standard inclusion $i:\upC^\bullet_{cb}(G;\bbR) \rightarrow \upC^\bullet_c(G;\bbR)$ induces a map at a cohomological level
$$
\textup{comp}^\bullet:\upH^\bullet_{cb}(G;\bbR) \rightarrow \upH^\bullet_c(G;\bbR) \ ,
$$
called \emph{comparison map}. 

Computing continuous bounded cohomology of a locally compact group $G$ using only the definition given above may reveal quite difficult. For this reason Burger and Monod \cite{burger2:articolo,monod:libro} introduced a way to compute continuous bounded cohomology groups based on the notion of resolutions. More precisely the authors showed \cite[Corollary 1.5.3]{burger2:articolo} that given any strong resolution $(E^\bullet,d^\bullet)$ of $\bbR$ by relatively injective Banach $G$-modules, it holds
$$
\upH^k_{cb}(G;\bbR) \cong \upH^k((E^\bullet)^G) \ ,
$$
for every $k \geq 0$. Since we will not need the notion of strong resolution and of relatively injective Banach $G$-module, we omit them and we refer to the book of Monod \cite{monod:libro} for more details. 

Unfortunately the isomorphism given above it is not isometric a priori, that is it may not preserve the seminormed structure. Nevertheless there are specific resolutions for which the isomorphism it is actually isometric. This is the case for instance when we consider the resolution of essentially bounded measurable functions $(\upL^\infty((G/Q)^{\bullet+1};\bbR),\delta^\bullet)$ on the quotient $G/Q$ \cite[Theorem 1]{burger2:articolo}, where $G$ is a semisimple Lie group of non-compact type and $Q \leq G$ is any amenable subgroup. 

Something relevant can be said also about the complex of bounded measurable functions $(\calB^\infty((G/Q)^{\bullet+1};\bbR),\delta^\bullet)$. Indeed the latter is a strong resolution of $E$ by \cite[Proposition 2.1]{burger:articolo}, together with the natural injection of coefficients. In this way, the projection on equivalence classes $\calB^\infty((G/Q)^{\bullet+1};\bbR) \rightarrow \textup{L}^\infty((G/Q)^{\bullet+1};\bbR)$ induces a well-defined map in cohomology
$$
\mathfrak{c}^\bullet:\textup{H}^\bullet(\calB^\infty((G/P)^{\bullet+1};\bbR)^G) \rightarrow \textup{H}^\bullet_{cb}(G;\bbR) \ , 
$$
as stated by Burger and Iozzi \cite[Proposition 2.2]{burger:articolo}. Hence any bounded measurable invariant cocycles naturally determines a class in the continuous bounded cohomology of $G$.

\subsection{Pullback maps induced by measurable cocycles} \label{sec:no:boundary:map}

In this brief section we are going to recall the pullback determined by a boundary map associated to a measurable cocycle.
We will actually introduce a more general way to define the pullback and we will show that it coincides with the approach introduced by the author and Moraschini \cite{moraschini:savini,moraschini:savini:2} when a boundary map exists. 

Let $\Gamma \leq G$ be a lattice of a semisimple Lie group of non-compact type. Let $H$ be another locally compact group and consider $(\Omega,\mu_\Omega)$ a standard Borel probability $\Gamma$-space. Given a measurable cocycle $\sigma:\Gamma \times \Omega \rightarrow H$, we can define the map
$$
\textup{C}^\bullet_b(\sigma):\textup{C}^\bullet_{cb}(H;\bbR) \rightarrow \textup{C}^\bullet_b(\Gamma;\bbR) \ ,
$$
$$
\psi \mapsto \textup{C}^\bullet_b(\sigma)(\psi)(\gamma_0,\ldots,\gamma_\bullet):=\int_\Omega \psi(\sigma(\gamma_0^{-1},s)^{-1},\ldots,\sigma(\gamma_\bullet^{-1},s)^{-1})d\mu_\Omega \ .
$$
Notice that the previous definition is motivated by the formula appearing in \cite[Theorem 5.6]{sauer:companion} and it is inspired by the cohomological induction introduced by Monod and Shalom \cite{MonShal} for measurable cocycles associated to couplings. 

\begin{lem}\label{lem:pullback:no:boundary}
For a measurable cocycle $\sigma:\Gamma \times \Omega \rightarrow H$, the map $\textup{C}^\bullet_b(\sigma)$ is a well-defined cochain map which restricts to the subcomplexes of invariant cochains
$$
\textup{C}^\bullet_b(\sigma):\textup{C}^\bullet_{cb}(H;\bbR)^H \rightarrow \textup{C}^\bullet_b(\Gamma;\bbR)^\Gamma \ ,
$$
and hence it induces a map in bounded cohomology 
$$
\textup{H}^\bullet_b(\sigma):\textup{H}^\bullet_{cb}(H;\bbR) \rightarrow \textup{H}^\bullet_b(\Gamma;\bbR) \ , \ \ \textup{H}^\bullet_b(\sigma)([\psi]):=\left[ \textup{C}^\bullet_b(\sigma)(\psi)\right] \ .
$$
\end{lem}

\begin{proof}
Since $\mu_\Omega$ is a probability measure, it is clear that $\textup{C}^\bullet_b(\sigma)$ preserves boundedness. The fact that $\textup{C}^\bullet_b(\sigma)$ is a cochain map is an easy computation that we leave to the reader. 

To conclude the proof we need to show that $\textup{C}^\bullet_b(\sigma)$ restricts to invariant cochains. Let $\psi \in \textup{C}^\bullet_{cb}(H;\bbR)^H$ be a $H$-invariant cochain. For any $\gamma,\gamma_0,\ldots,\gamma_\bullet \in \Gamma$ it holds

\begin{align*}
\gamma \cdot \textup{C}^\bullet_b(\sigma)(\psi)(\gamma_0,\ldots,\gamma_\bullet)&=\textup{C}^\bullet_b(\sigma)(\psi)(\gamma^{-1} \gamma_0,\ldots,\gamma^{-1}\gamma_\bullet)= \\
&=\int_\Omega \psi(\sigma(\gamma_0^{-1}\gamma,s)^{-1},\ldots,\sigma(\gamma_\bullet^{-1}\gamma,s)^{-1})d\mu_\Omega(s)=\\
&=\int_\Omega \psi(\sigma(\gamma,s)^{-1}\sigma(\gamma_0^{-1},\gamma s)^{-1},\ldots,\sigma(\gamma,s)^{-1}\sigma(\gamma_\bullet^{-1},\gamma s)^{-1})d\mu_\Omega(s)=\\
&=\int_\Omega \psi(\sigma(\gamma_0^{-1},s)^{-1},\ldots,\sigma(\gamma_\bullet^{-1},s)^{-1})d\mu_\Omega(s)=\textup{C}^\bullet_b(\sigma)(\psi)(\gamma_0,\ldots,\gamma_\bullet) \ ,
\end{align*}
where we used Equation \eqref{eq:measurable:cocycle} to move from the second line to the third one and we exploited jointly the $H$-invariance of $\psi$ and the $\Gamma$-invariance of $\mu_\Omega$ to move from the third line to the fourth one. This concludes the proof.
\end{proof}

Thank to Lemma \ref{lem:pullback:no:boundary} we can give the following

\begin{deft}\label{def:pullback:measurable:cocycle}
Let $\Gamma \leq G$ be a lattice in a semisimple Lie group of non-compact type and let $(\Omega,\mu_\Omega)$ be a standard Borel probability $\Gamma$-space. Given a measurable cocycle $\sigma:\Gamma \times \Omega \rightarrow H$ with values into a locally compact group, we define the \emph{pullback induced by $\sigma$} as the map 
$$\textup{H}^\bullet_b(\sigma):\textup{H}^\bullet_{cb}(H;\bbR) \rightarrow \textup{H}^\bullet_b(\Gamma;\bbR) \ . $$ 
\end{deft}

It is quite natural to ask how the pullback map defined above varies along the cohomology class of a fixed measurable cocycle. We are going to show that it is actually constant. 

\begin{lem}\label{lem:pullback:cohomology:class}
Let $\Gamma \leq G$ be a lattice in a semisimple Lie group of non-compact type and let $(\Omega,\mu_\Omega)$ be a standard Borel probability $\Gamma$-space. Let $f:\Omega \rightarrow H$ be a measurable function with values into a locally compact group. Given a measurable cocycle $\sigma:\Gamma \times \Omega \rightarrow H$, it holds that
$$
\textup{H}^\bullet_b(\sigma^f)=\textup{H}^\bullet_b(\sigma) \ ,
$$
where $\sigma^f$ is the $f$-twisted cocycle associated to $\sigma$. 
\end{lem}

\begin{proof}
We are going to follow the line of \cite[Lemma 8.7.2]{monod:libro}. The main goal to prove the statement is to find a chain homotopy between $\textup{C}^\bullet_b(\sigma)$ and $\textup{C}^\bullet_b(\sigma^f)$ to show that they determine the same map in cohomology. 

Let $\psi \in \textup{C}_{cb}^\bullet(H;\bbR)^H$. We have
\begin{align}\label{eq:cochain:sigmaf}
\textup{C}^\bullet_{cb}(\sigma^f)(\psi)(\gamma_0,\ldots,\gamma_\bullet)&=\int_\Omega \psi(\sigma^f(\gamma_0^{-1},s)^{-1},\ldots,\sigma^f(\gamma_\bullet^{-1},s)^{-1})d\mu_\Omega=\\
&=\int_\Omega \psi(f(s)^{-1}\sigma(\gamma_0^{-1},s)^{-1}f(\gamma_0^{-1}s),\dots )d\mu_\Omega = \nonumber \\
&=\int_\Omega \psi(\sigma(\gamma_0^{-1},s)^{-1}f(\gamma_0^{-1}s),\ldots )d\mu_\Omega \nonumber \ ,
\end{align}
where we moved from the first line to the second one using the definition of $\sigma^f$ and we exploited the $H$-invariance of $\psi$ to move from the second line to the third one.  

For $0 \leq i \leq \bullet-1$ we now define the following map
$$
s^\bullet_i(\sigma,f):\textup{C}^\bullet_{cb}(H;\bbR) \rightarrow \textup{C}^{\bullet - 1}_b(\Gamma;\bbR) \ , \ \ s^\bullet_i(\sigma,f)(\psi)(\gamma_0,\ldots,\gamma_{\bullet-1}):=
$$
$$
=\int_\Omega \psi(\sigma(\gamma_0^{-1},s)^{-1}f(\gamma_0^{-1}s),\ldots,\sigma(\gamma_i^{-1},s)^{-1}f(\gamma_i^{-1}s),\sigma(\gamma_i^{-1},s)^{-1},\ldots,\sigma(\gamma_{\bullet-1}^{-1},s)^{-1})d\mu_\Omega(s) \ ,
$$
and we set $s^\bullet(\sigma,f):=\sum_{i=0}^{\bullet-1}(-1)^i s^\bullet_i(\sigma,f)$. By defining for $-1 \leq i \leq \bullet$ the map
$$
\rho^\bullet_i(\sigma,f):\textup{C}^\bullet_{cb}(H;\bbR) \rightarrow \textup{C}^{\bullet}_b(\Gamma;\bbR) \ , \ \ 
\rho^\bullet_i(\sigma,f)(\psi)(\gamma_0,\ldots,\gamma_\bullet):=
$$
$$
=\int_\Omega \psi(\sigma(\gamma_0^{-1},s)^{-1}f(\gamma_0^{-1}s),\ldots,\sigma(\gamma_i^{-1},s)^{-1}f(\gamma^{-1}_i s),\sigma(\gamma^{-1}_{i+1},s)^{-1},\ldots,\sigma(\gamma^{-1}_\bullet,s)^{-1})d\mu_\Omega(s) \ ,
$$
we can notice that $\rho^\bullet_{-1}(\sigma,f)=\textup{C}^\bullet_b(\sigma)$. Following the same computation of \cite[Lemma 8.7.2]{monod:libro}, we get that 
\begin{align*}
s^{\bullet+1}(\sigma,f)\delta^{\bullet}&=-\delta^{\bullet}s^{\bullet}(\sigma,f)+\sum_{i=0}^{\bullet}(\rho_{i-1}^{\bullet}(\sigma,f)-\rho^{\bullet}_i)=\\
&=-\delta^{\bullet}s^{\bullet}(\sigma,f)+\textup{C}^\bullet_b(\sigma)- \rho^\bullet_\bullet(\sigma,f) \ ,
\end{align*}
where $\delta^\bullet$ is the usual homogeneous coboundary operator. Since by Equation \eqref{eq:cochain:sigmaf} on the subcomplex of $H$-invariants cochains it holds
$$
\rho^\bullet_\bullet(\sigma,f)=\textup{C}^\bullet_b(\sigma^f) \ ,
$$
we get that 
$$
s^{\bullet+1}(\sigma,f)\delta^{\bullet}+\delta^{\bullet}s^{\bullet}(\sigma,f)=\textup{C}^\bullet_b(\sigma)- \textup{C}^\bullet_b(\sigma^f) \ ,
$$
and the claim follows. 
\end{proof}

We want now to relate Definition \ref{def:pullback:measurable:cocycle} with the approach followed by the author and Moraschini in \cite{moraschini:savini,moraschini:savini:2}. In the same setting of Definition \ref{def:pullback:measurable:cocycle}, consider a minimal parabolic subgroup $P \leq G$. Let $(Y,\nu)$ be any measure space such that the group $H$ acts on $Y$ by preserving the measure class of $\nu$. Given a boundary map $\phi:G/P \times \Omega \rightarrow Y$ associated to a measurable cocycle $\sigma:\Gamma \times \Omega \rightarrow H$, there exists a natural map defined at the level of cochains as
$$
\upC^\bullet(\Phi^\Omega):\calB^\infty(Y^{\bullet+1};\bbR)^H \rightarrow \upL^\infty((G/Q)^{\bullet+1};\bbR)^\Gamma \ ,
$$
$$
\upC^\bullet(\Phi^\Omega)(\psi)(\xi_0,\ldots,\xi_\bullet):=\int_{\Omega} \psi(\phi(\xi_0,s),\ldots,\phi(\xi_\bullet,s))d\mu_\Omega(s) \ ,
$$
where $\textup{C}^\bullet(\Phi^\Omega)(\psi)$ has to be intended as an $\textup{L}^\infty$-equivalence class. As shown by the author and Moraschini \cite{moraschini:savini,moraschini:savini:2}, the above map is a chain map which does not increase the norm and it induces a well-defined map in cohomology 
$$
\upH^\bullet(\Phi^\Omega):\upH^\bullet(\calB(Y^{\bullet+1};\bbR)^H) \rightarrow \upH^\bullet_{cb}(G;\bbR) \ , \ \upH^\bullet(\Phi^\Omega)([\psi]):=[\upC^\bullet(\Phi^\Omega)(\psi)] \ . 
$$
We are going to call the map $\textup{H}^\bullet(\Phi^\Omega)$ \emph{pullback induced by the boundary map} $\phi$.  We have the following result which should be interpreted as an extension of \cite[Corollary 2.7]{burger:articolo}.

\begin{lem}\label{lem:pullback:cocycle:boundary}
Let $\Gamma \leq G$ be a lattice in a semisimple Lie group of non-compact type. Consider a minimal parabolic subgroup $P \leq G$ and a standard Borel probability $\Gamma$-space $(\Omega,\mu_\Omega)$. Let $(Y,\nu)$ be a measure space on which a locally compact group $H$ acts by preserving the measure class of $\nu$. Suppose that a measurable cocycle $\sigma:\Gamma \times X \rightarrow H$ admits a boundary map $\phi:G/P \times \Omega \rightarrow Y$. Given $\psi \in \calB^\infty(Y^{\bullet+1};\bbR)^H$, then 
$$
\textup{C}^\bullet(\Phi^\Omega)(\psi) \in \textup{L}^\infty((G/P)^{\bullet+1};\mathbb{R})^\Gamma \ ,
$$  
is a natural representative of the class $\textup{H}^\bullet_b(\sigma)([\psi]) \in \textup{H}^\bullet_{cb}(\Gamma;\bbR)$. 
\end{lem}

\begin{proof}
It is sufficient to apply \cite[Proposition 1.2]{burger:articolo} to get the following commutative diagram
$$
\xymatrix{
\textup{H}^\bullet(\calB^\infty(Y^{\bullet+1};\bbR)^H) \ar[dd]^{\mathfrak{c}^\bullet} \ar[rrr]^{\textup{H}^\bullet(\Phi^\Omega)} &&& \textup{H}^\bullet_b(\Gamma;\bbR) \\
\\
\textup{H}^\bullet_{cb}(H;\bbR) \ar[uurrr]^{\textup{H}^\bullet_b(\sigma)} &&& \ ,
}
$$
and the statement follows. 
\end{proof}

We are going to use the pullback maps introduced so far to define properly the Toledo invariant of a measurable cocycle of a surface group.

\subsection{Lie groups of Hermitian type}\label{sec:hermitian:groups}

In this section we are going to recall the main definitions and results about Lie groups of Hermitian type. We are going to remind the notion of Shilov boundary for a Hermitian symmetric space and we are going to define a suitable cocycle on it, called Bergmann cocycle, which will enable us to define the notion of maximality for measurable cocycles of surface groups. For a more detailed discussion about these notions, we refer mainly to the work of Burger, Iozzi and Wienhard \cite{BIW07,BIW09,BIW1}. 

\begin{deft}\label{def:hermitian:symmetric:space}
Let $\calX$ be a Riemannian symmetric space and denote by $G=\textup{Isom}(\calX)^\circ$ the connected component of the identity of the isometry group associated to $\calX$. We say that $\calX$ is \emph{Hermitian} if there exists a $G$-invariant complex structure $\calJ$ on $\calX$. Given a semisimple real algebraic group $\mathbf{G}$, we say that $G=\mathbf{G}(\bbR)^\circ$ is \emph{of Hermitian type} if its symmetric space $\calX$ is Hermitian. 
\end{deft}

Among all the possible ones, a family of examples of particular interest in this paper will be the one of Hermitian symmetric spaces of tube type. We say that a Hermitian symmetric space $\calX$ is \emph{of tube type} if it is biholomorphic to a complex subset of the form $V+i\Omega$, where $V$ is a real vector space and $\Omega \subset V$ is a proper convex cone. The prototypical example is given by the upper half plane model for the complex hyperbolic line $\bbH^1_{\bbC}$. The latter is the symmetric space associated to the group $\pu(1,1)$, and more generally is of tube type the symmetric space associated to $\pu(p,p)$ when $p \geq 2$. 

A Hermitian symmetric space $\calX$ can be bihomolorphically realized as bounded convex domain $\calD_{\calX}$ in $\bbC^n$. For such a realization, the group $G=\textup{Isom}(\calX)^\circ$ acts via biholomorphisms and its action can be extended in a continuous way to the boundary $\partial \calD_{\calX}$. Unfortunately the latter is not a homogeneous $G$-space, but it admits a unique closed $G$-orbit. The latter will be identified with the Shilov boundary.

More precisely we give first the following

\begin{deft}\label{def:shilov:boundary}
Let $\calD \subset \bbC^n$ be a bounded domain. The \emph{Shilov boundary} $\check{\calS}_{\calD}$ of $\calD$ is the unique closed subset of $\partial \calD$ such that, given a function $f$ continuous on $\overline{\calD}$ and holomorphic on $\calD$, then 
$$
\max_{\overline{D}}|f|=\max_{\check{\calS}_{\calD}} |f| \ . 
$$
Given a Hermitian symmetric space $\calX$, we denote by $\check{\calS}_{\calX}$ the Shilov boundary associated to the bounded realization of $\calX$ and we call it \emph{the Shilov boundary of $\calX$}.
\end{deft}

As already anticipated the Shilov boundary associated to a Hermitian symmetric space $\calX$ is a homogeneous $G$-space. Indeed if we denote by $\mathbf{G}$ the algebraic group associated to the complexified Lie algebra of $G=\textup{Isom}(\calX)^\circ$, then there exists a maximal parabolic subgroup $\mathbf{Q} \subset \mathbf{G}$ such that $\check{\calS}_{\calX}$ can be identified with $(\mathbf{G}/\mathbf{Q})(\bbR)$. Such an identification determines naturally a structure of algebraic variety on $\check{\calS}_{\calX}$. 

Although $\check{\calS}_{\calX}$ is an amenable $G$-space only when the rank of $G$ is equal to 1, we can use the resolution of bounded measurable functions $(\calB^\infty((\check{\calS}_{\calX})^{\bullet+1};\bbR),\delta^\bullet)$ to obtain in a canonical way a class in the continuous bounded cohomology of $G$, as noticed in Section \ref{sec:burger:monod}. We are going to focus our attention to the particular case when the degree is equal to $2$. In order to describe more accurately the second bounded cohomology group of $G$, recall that if $\calX$ is Hermitian, then there exists a $G$-invariant complex structure $\calJ$ on it. If we denote by $g$ the $G$-invariant Riemannian metric on $\calX$, we can define the \emph{K\"ahler form} at $a \in \calX$ as
$$
(\omega_{\calX})_a(X,Y):=g_a(X,\calJ_a Y) \ ,
$$
for any $X,Y \in T_a \calX$. Being $G$-invariant, the form $\omega_{\calX}$ is automatically closed by Cartan's Lemma \cite[VII.4]{Hel01}. Define now
\begin{equation}\label{eq:cocycle:symmetric:space}
\beta_{\calX}: (\calX)^{(3)} \rightarrow \bbR, \ \ \beta_{\calX}(a_1,a_2,a_3):= \frac{1}{2\pi} \int_{\Delta(a_1,a_2,a_3)} \omega_{\calX} \ ,
\end{equation} 
where $\Delta(a_1,a_2,a_3)$ is any triangle with geodesic sides determined by $a_1,a_2,a_3 \in \calX$. Since $\omega_{\calX}$ is closed, by Stokes' Theorem the function $\beta_{\calX}$ is an alternating $G$-invariant cocycle on $\calX$. Remarkably the cocycle $\beta_{\calX}$ can be extended to a strict measurable $G$-invariant cocycle on the Shilov boundary $\check{\calS}_{\calX}$ \cite[Corollary 3.8]{BIW07} and its absolute value is bounded by $\frac{\rk(\calX)}{2}$. We are going to denote such an extension with $\beta_{\calX}$ with an abuse of notation. 
As previously said in Section \ref{sec:burger:monod} the cocycle $\beta_{\calX} \in \calB^\infty((\check{\calS}_{\calX})^{(3)};\bbR)^G$ determines canonically a class in $\upH^2_{cb}(G;\bbR)$. 

\begin{deft}
We call \emph{Bergmann cocycle} the measurable extension $\beta_{\calX}: \check{\calS}^{(3)} \rightarrow \bbR$ to the Shilov boundary of the cocycle defined by Equation \eqref{eq:cocycle:symmetric:space}. 

We denote by $\kappa^b_G \in \upH^2_{cb}(G;\bbR)$ the class determined by $\beta_{\calX}$ and we call it \emph{bounded K\"ahler class}. 
\end{deft}

Recall that two points $\xi,\eta \in \check{\calS}_{\calX}$ are \emph{transverse} if they lie in the unique open $G$-orbit in $(\check{\calS}_{\calX})^2$. A triple of points $(\xi,\eta,\omega)$ will be said \emph{maximal} if it satisfies $|\beta_{\calX}(\xi,\eta,\omega)|=\frac{\rk(\calX)}{2}$. We conclude the section by recalling some properties of the Bergmann cocycle when $\calX$ is a Hermitian symmetric space of tube type. As stated in \cite[Lemma 5.5]{BIW1}, if $\calX$ is of tube type then
\begin{enumerate}
\item the cocycle $\beta_{\calX}$ takes values in the discrete set $$\{-\frac{\rk(\calX)}{2}, - \frac{\rk(\calX)}{2}+1 , \ldots , \frac{\rk(\calX)}{2}-1, \frac{\rk(\calX)}{2} \} \ ; $$
\item if the triple $(\xi,\eta,\omega)$ is maximal, then $\xi,\eta,\omega$ are pairwise transverse;
\item we can decompose 
$$
(\check{\calS}_{\calX})^{(3)}= \sqcup_{i=0}^{\rk(\calX)} \calO_{-\rk(\calX)+2i} \ ,
$$
where $\calO_{-\rk(\calX)+2i}$ is the open subset of $(\check{\calS}_{\calX})^{(3)}$ where $\beta_{\calX}$ is identically equal to $-\frac{\rk(\calX)}{2}+i$;
\item given $\xi, (\xi_n)_{n \in \bbN}, (\xi'_n)_{n \in \bbN}$ where $\xi,\xi_n,\xi'_n \in \check{\calS}_{\calX}$, if $\lim_{n \to \infty} \xi_n =\xi$ and the triple $(\xi,\xi_n,\xi_n')$ is maximal for every $n \in \bbN$, then $\lim_{n \to \infty} \xi'_n=\xi$. 
\end{enumerate}

\section{Maximal measurable cocycles of surface groups}\label{sec:maximal:cocycles}

In this section we are going to introduce the definition and the main properties of the \emph{Toledo invariant} of a measurable cocycle associated to a surface group. Such invariant will have bounded absolute value and we will see that in case of maximality one can get important information about the target group of the measurable cocycle. 

\subsection{The Toledo invariant of a measurable cocycle}\label{sec:toledo:invariant}
Let $L$ be a finite connected covering of the group $\pu(1,1)$ and consider a torsion-free lattice $\Gamma \leq L$. Let $(\Omega,\mu_\Omega)$ be a standard Borel probability $\Gamma$-space. Denote by $G=\text{Isom}^\circ(\calX)$ the connected component of the identity of the isometry group of an irreducible Hermitian symmetric space $\calX$.
Let $\mathbf{G}$ be the connected Lie group associated to the complexified Lie algebra of $G$, so that $G=\mathbf{G}(\bbR)^\circ$. Let $\sigma:\Gamma \times \Omega \rightarrow G$ be a measurable cocycle. 

Since we are exactly in the situation described by Definition \ref{def:pullback:measurable:cocycle}, we have a pullback map in cohomology 
$$
\upH^\bullet_b(\sigma): \upH^\bullet_{cb}(G;\bbR) \rightarrow \upH^\bullet_b(\Gamma;\bbR) \ .
$$
In particular we are allowed to consider the pullback of the bounded K\"ahler class $\kappa^b_G$. Being $\Gamma$ a lattice of $L$, we have a well-defined \emph{transfer map}. Since we are going to use it later, we are going to recall its definition using the resolution of essentially bounded functions. The transfer map is given at the level of cochains by
$$
\hat{\upT}^\bullet_b:\upL^\infty((\bbS^1)^{\bullet+1};\bbR)^\Gamma \rightarrow \upL^\infty((\bbS^1)^{\bullet+1};\bbR)^L \ ,
$$
$$
\hat{\upT}^\bullet_b(\psi)(\xi_0,\ldots,\xi_\bullet):=\int_{\Gamma \backslash L} \psi(\overline{g}\xi_0,\ldots,\overline{g}\xi_\bullet)d\mu_{\Gamma \backslash L}(\overline{g}) \ ,
$$
where $\overline{g}$ denotes the equivalence class of $g$ in $\Gamma \backslash L$ and $\mu_{\Gamma \backslash L}$ is the normalized $L$-invariant measure on the quotient. Being a chain map, $\hat \upT^\bullet_b$ induces a well-defined map in cohomology called \emph{transfer map}
$$
\upT^\bullet_b:\upH^\bullet_b(\Gamma;\bbR) \rightarrow \upH^\bullet_{cb}(L;\bbR), \hspace{5pt} \upT^\bullet_b([\psi]):=[\hat \upT^\bullet_b(\psi)] \ .
$$

It is worth recalling that the bounded K\"ahler class $\kappa^b_L$ is a generator of the group $\upH^2_{cb}(L;\bbR)$ which is indeed a one dimensional vector space. In this particular setting, we are allowed to give the following 

\begin{deft}\label{def:toledo:inv}
Let $\Gamma \leq L$ be a torsion-free lattice and $(\Omega,\mu_\Omega)$ a standard Borel probability $\Gamma$-space. Consider a measurable cocycle $\sigma:\Gamma \times \Omega \rightarrow G$. The \emph{Toledo invariant} $\textup{t}_b(\sigma)$ associated to $\sigma$ is defined as
\begin{equation}\label{eq:toledo:cocycle}
\upT^2_b \circ \upH^2_b(\sigma)(\kappa^b_G)=\textup{t}_b(\sigma)\kappa^b_L \ .
\end{equation}
\end{deft}

The first natural question that one could ask is how the Toledo number varies along the cohomology class of $G$. We are going to prove that it is constant along that class and it has bounded absolute value. 

\begin{lem}\label{lem:toledo:invariant}
Let $\Gamma \leq L$ be a torsion-free lattice and $(\Omega,\mu_\Omega)$ a standard Borel probability $\Gamma$-space. Consider a measurable cocycle $\sigma:\Gamma \times \Omega \rightarrow G$. Then the Toledo invariant depends only on the $G$-cohomology class of $\sigma$ and it holds
$$
|\textup{t}_b(\sigma)| \leq \rk(\calX) \ .
$$
\end{lem}

\begin{proof}
The fact that the Toledo invariant is constant on the $G$-cohomology class is a direct consequence of Lemma \ref{lem:pullback:cohomology:class}. 

To prove the boundedness of the absolute value, recall that both $\textup{T}^2_b$ and $\textup{H}^2_b(\sigma)$ are norm non-increasing maps. Thus it follows that 
$$
\lVert \textup{t}_b(\sigma) \kappa^b_L \rVert_\infty = \lVert \upT^2_b \circ \upH^2_b(\sigma)(\kappa^b_G) \rVert_\infty =\frac{\rk (\calX)}{2} \ , 
$$
and since it holds $\lVert \kappa^b_L \rVert_\infty=\frac{1}{2}$, the claim follows.
\end{proof}

Suppose now that $\sigma:\Gamma \times \Omega \rightarrow G$ admits a boundary map $\phi:\mathbb{S}^1 \times \Omega \rightarrow \check{\calS}_{\calX}$, where $\check{\calS}_{\calX}$ is the Shilov boundary of the Hermitian symmetric space $\calX$. From Section \ref{sec:no:boundary:map} we know that the Toledo invariant can be computed alternatively with the use of the boundary map $\phi$. Indeed if we consider the pullback along $\phi$, namely
$$
\textup{H}^2(\Phi^\Omega):\textup{H}^2(\calB^\infty((\check{\calS}_{\calX})^3;\bbR)^G) \rightarrow \textup{H}^2_b(\Gamma;\bbR) \ ,
$$
then we know by Lemma \ref{lem:pullback:cocycle:boundary} that it holds
$$
\textup{H}^2_b(\sigma)(\kappa^b_G)=\textup{H}^2(\Phi^\Omega)(\left[ \beta_{\calX} \right]) \ ,
$$
where we used the fact that the Bergmann cocycle $\beta_{\calX}$ is a canonical representative of the bounded K\"ahler class
$\kappa^b_G$. As a consequence, Equation \eqref{eq:toledo:cocycle} can be alternatively rewritten as follows
\begin{equation}\label{eq:toledo}
\textup{T}^2_b \circ \textup{H}^2(\Phi^\Omega)(\left[ \beta_{\calX} \right])=\textup{t}_b(\sigma)\kappa^b_L \ .
\end{equation}

Recall now that $\kappa^b_L$ can be representend by $\beta_{\bbS^1}$ (which is nothing else that the \emph{orientation cocycle} up to a factor equal to $\frac{1}{2}$). Since the $\Gamma$-action on the circle is doubly ergodic and the cocycles that we are considerng are alternating, Equation \eqref{eq:toledo} holds actually at the level of bounded measurable cochains, that is 
\begin{align}\label{eq:formula}
\int_{\Gamma \backslash L} \int_\Omega & \beta_{\calX}(\phi(\overline{g}\xi,s),\phi(\overline{g}\eta,s),\phi(\overline{g}\omega,s))d\mu_\Omega(s)d\mu_{\Gamma \backslash L}(\overline{g}) =\\
=& \text{t}_b(\sigma)\beta_{\bbS^1}(\xi,\eta,\omega) \nonumber \  ,
\end{align}
and the equation holds for \emph{every triple of pairwise distinct points} $\xi,\eta,\omega \in \bbS^1$, as a consequence of either Burger and Iozzi \cite{BIcartan} or Pozzetti \cite{Pozzetti}, for instance. Notice that Equation \eqref{eq:formula} is simply a suitable adaptation of \cite[Corollary 4.4]{BIW1} to the context of measurable cocycles. Indeed, suppose that $\sigma$ is a measurable cocycle associated to a Zariski dense representation $\rho:\Gamma \rightarrow G$. Then the boundary map of $\sigma$ will be nothing else that the boundary map associated to $\rho$ (which exists by \cite[Proposition 7.2]{BI04}) and the formula given by Equation \eqref{eq:formula} boils down to \cite[Corollary 4.4]{BIW1}.

It is immediate to verify that the Toledo invariant is a \emph{multiplicative constant} in the sense of \cite[Definition 3.16]{moraschini:savini:2}. Indeed following the notation of that paper, the setting required by \cite[Definition 3.16]{moraschini:savini:2} is satisfied and one has
$$
\textup{t}_b(\sigma)=\lambda_{\beta_{\calX},\beta_{\bbS^1}}(\sigma) \ ,
$$
where $\lambda_{\beta_{\calX},\beta_{\bbS^1}}(\sigma)$ denotes the multiplicative constant associated the measurable cocycle $\sigma$ for the Bergmann cocycles $\beta_{\calX},\beta_{\bbS^1}$.

\begin{oss}\label{oss:alternative:definition}
We could have defined the Toledo invariant in a different way. Let $\Gamma \leq L$ be a torsion-free lattice and let $(\Omega,\mu_\Omega)$ be a standard Borel probability $\Gamma$-space. Denote by $\Sigma$ the finite-area surface obtained as the quotient of $\bbH^2_{\bbR}$ by $\Gamma$, that is $\Sigma=\Gamma \backslash \bbH^2_{\bbR}$. If $\Gamma$ is \emph{uniform} we know that $\Sigma$ is closed, whereas when $\Gamma$ is \emph{non-uniform} then the surface $\Sigma$ has finitely many cusps. In the latter case we are going to denote by $S$ a \emph{compact core} of $\Sigma$, otherwise we set $S=\Sigma$. 

Following \cite[Section 3.4]{moraschini:savini} we can define the following composition of functions
\begin{equation}\label{eq:j:composition}
\upJ^\bullet_{S, \partial S}: \upH^\bullet_b(\Gamma;\bbR) \rightarrow \upH^\bullet_b(\Sigma;\bbR) \rightarrow \upH^\bullet_b(\Sigma,\Sigma \setminus S) \rightarrow \upH^\bullet_b(S,\partial S) \ ,
\end{equation}
where the first map is the isomorphism given by the Gromov's Mapping Theorem \cite{Grom82,Ivanov,FM:grom}, the second map is obtained by the long exact sequence in bounded cohomology \cite{BBFIPP} and the last map is induced by the homotopy equivalence $(\Sigma,\Sigma \setminus S) \simeq (S, \partial S)$. 

Given a measurable cocycle $\sigma:\Gamma \times \Omega \rightarrow G$, we could have defined the \emph{Toledo number} of the cocycle $\sigma$ as
$$
\upT_b(\sigma):= \langle \textup{comp}^2_{S, \partial S} \circ \upJ^2_{S, \partial S} \circ \upH^2_b(\sigma)(\kappa^b_G),[S,\partial S] \rangle \ ,
$$
where we denoted by 
$$\textup{comp}^2_{S, \partial S}:\textup{H}^2_b(S,\partial S) \rightarrow \textup{H}^2(S,\partial S) \ ,$$
the \emph{comparison map} associated to the pair $(S,\partial S)$. 

To compare the two different definitions of the Toledo invariant, one can follows the same strategy of the proofs of either \cite[Proposition 1.2, Proposition 1.6]{moraschini:savini} or \cite[Proposition 5.5]{moraschini:savini:2}. It is worth mentioning that the same idea contained in the proofs of those propositions can be actually adapted also to the case when $\sigma$ does not admits a boundary. In this way it is possible to show that
\begin{equation}\label{eq:alternative:toledo}
\textup{t}_b(\sigma)=\frac{\textup{T}_b(\sigma)}{|\chi(\Sigma)|} \ ,
\end{equation}
where $\chi(\Sigma)$ is the Euler characteristic of the surface $\Sigma$. Notice that Equation \eqref{eq:alternative:toledo} is analogous to the one obtained by Burger, Iozzi and Wienhard \cite[Theorem 3.3]{BIW1}. In particular $\textup{T}_b(\sigma)$ is an invariant of the $G$-cohomology class of $\sigma$ and it holds the following estimate
$$
|\textup{T}_b(\sigma)| \leq \rk(\calX) |\chi(\Sigma)| \ . 
$$
\end{oss}

\subsection{Maximal measurable cocycles of surface groups}\label{sec:maximal:cocycle:thm}

In this section we are going to introduce the notion of maximality. Maximal measurable cocycles represent the first example of tight cocycles and this has important consequences on their algebraic hull. 

We start by giving the definition of maximality. 

\begin{deft}\label{def:maximal:cocycle}
Let $\Gamma \leq L$ be a torsion-free lattice and let $(\Omega,\mu_\Omega)$ be a standard Borel probability $\Gamma$-space. Consider a measurable cocycle $\sigma:\Gamma \times \Omega \rightarrow G$. We say that $\sigma$ is \emph{maximal} if it holds $\text{t}_b(\sigma)=\rk(\calX)$. 
\end{deft}

In order to show that maximal cocycles are tight, we need first to introduce the notion of tightness for measurable cocycles of surface groups. Inspired by the notion for representations studied by Burger, Iozzi and Wienhard \cite{BIW09}, we can give the following 

\begin{deft}\label{def:tight:cocycle}
Let $\Gamma \leq L$ be a torsion-free lattice and $(\Omega,\mu_\Omega)$ a standard Borel probability $\Gamma$-space. Consider a measurable cocycle $\sigma:\Gamma \times \Omega \rightarrow G$. We say that $\sigma$ is \emph{tight} if it holds
$$
\lVert \upH^2_b(\sigma)(\kappa^b_G) \rVert_\infty=\frac{\rk(\calX)}{2} \ .
$$
\end{deft}

\begin{oss}
It is worth mentioning that in the particular case when $\sigma$ admits a boundary map $\phi:\bbS^1 \times \Omega \rightarrow \check{\calS}_{\calX}$, then the notion of tightness can be alternatively rewritten as follows
$$
\lVert \textup{H}^2(\Phi^\Omega)([\beta_{\calX}]) \rVert_\infty=\frac{\rk(\calX)}{2} \ .
$$
\end{oss}

Clearly the definition above mimic the one given for representations. Indeed it is immediate to check that if the cocycle is cohomologous to the one induced by a representation, Definition \ref{def:tight:cocycle} boils down to the standard one. Another important aspect is that the tightness property is invariant along the $G$-cohomology class of a given cocycle by Lemma \ref{lem:pullback:cohomology:class}. Notice that we could have introduced the notion of tightness in a much more general setting, but this would be not so useful for our purposes. 

The deep study of tight representations done by Burger, Iozzi and Wienhard \cite{BIW09} enables us to state the following theorem which characterizes the algebraic hull of a tight cocycle and which is a direct consequence of \cite[Theorem 3]{BIW09}, where a full characterization of tight subgroups is given. 

\begin{teor}\label{teor:alg:hull:tight}
Let $\Gamma$ be a torsion-free lattice of a finite connected covering $L$ of $\pu(1,1)$ and let $(\Omega,\mu_\Omega)$ be a standard Borel probability $\Gamma$-space. Consider $\mathbf{G}$ a semisimple algebraic $\bbR$-group such that $G=\mathbf{G}(\bbR)^\circ$ is a Lie group of Hermitian type. Given a measurable cocycle $\sigma:\Gamma \times \Omega \rightarrow G$, denote by $\mathbf{H}$ the algebraic hull of $\sigma$ in $\mathbf{G}$ and set $H=\mathbf{H}(\bbR)^\circ$. If $\sigma$ is tight then 
\begin{enumerate}
	\item $\mathbf{H}$ is a reductive group;
	\item the centralizer $Z_G(H)$ is compact; 
	\item if $\calY$ is the symmetric space associated to $H$, there exists a unique $H$-invariant complex structure on $\calY$ such that the inclusion $H \rightarrow G$ is tight and positive. 
\end{enumerate}
\end{teor}

\begin{proof}
Since the cocycle is tight and this condition is invariant along the $G$-cohomology class of $\sigma$, the inclusion $i:H \rightarrow G$ is tight. The conclusion follows by direct application of \cite[Theorem 7.1]{BIW09} which characterize tight subgroups of $G$. 
\end{proof}

The next step is to prove that maximal cocycles are tight in the sense of Definition \ref{def:tight:cocycle}, similarly for what happens in the case of representations \cite[Lemma 6.2]{BIW1}. This result will have important consequence for the algebraic hull of a maximal cocycle as a direct application of Theorem \ref{teor:alg:hull:tight}.

\begin{prop}\label{prop:maximal:tight}
Let $\Gamma \leq L$ be a torsion-free lattice and let $(\Omega,\mu_\Omega)$ be a standard Borel probability $\Gamma$-space. Consider a measurable cocycle $\sigma:\Gamma \times \Omega \rightarrow G$. If $\sigma$ is maximal then it is tight. 
\end{prop}

\begin{proof}
Suppose that $\sigma:\Gamma \times \Omega \rightarrow G$ is maximal. Then it holds $\textup{t}_b(\sigma)=\rk\calX$. By definition we have that 
$$
\textup{T}^2_b \circ \textup{H}^2_b(\sigma)(\kappa^b_G)=\rk(\calX) \kappa^b_L \ ,
$$
and hence it follows
$$
\frac{\rk(\calX)}{2} = \lVert \rk(\calX) \kappa^b_L \rVert_\infty=\lVert \textup{T}^2_b \circ \textup{H}^2_b(\sigma)(\kappa^b_G) \rVert \leq \rVert \textup{H}^2_b(\sigma)(\kappa^b_G) \rVert_\infty \ .
$$
Since the pullback is norm non-increasing, we have also that $\lVert \textup{H}^2_b(\sigma)(\kappa^b_G) \rVert_\infty \leq \frac{\rk(\calX)}{2}$, whence we must have equality and the cocycle $\sigma$ is tight. 
\end{proof} 

\section{Maximality and Zariski density}\label{sec:zariski:maximal}

In this section we will focus our attention to maximal measurable cocycles of surface groups which are Zariski dense. Under the necessary assumption of the existence of a boundary map we are going to show that the symmetric space associated to their algebraic hull is of \emph{tube type}. We will study also some regularity properties of the boundary map.  

\subsection{Algebraic hull and Zariski density}

Let $\Gamma \leq L$ be a torsion-free lattice in a finite connected covering of $\pu(1,1)$ and let $(\Omega,\mu_\Omega)$ be a standar Borel probability $\Gamma$-space. Consider a semisimple real algebraic group $\mathbf{G}$ such that $G:=\mathbf{G}(\bbR)^\circ$ is of Hermitian type. Suppose to have a maximal measurable cocycle $\sigma:\Gamma \times \Omega \rightarrow G$. Up to twisting $\sigma$ and restricting the image to its algebraic hull, we will assume that $\sigma$ is Zariski dense. 

Since in this section we will have to work with Equation \eqref{eq:formula}, we will need to assume that there exists a boundary map $\phi:\bbS^1 \times \Omega \rightarrow \check{\calS}_{\calX}$ associated to $\sigma$. 

It is quite natural to ask under which condition a boundary map exists. In the particular case that $G=\textup{PU}(1,1)$, every \emph{non-elementary} measurable cocycle $\sigma:\Gamma \times \Omega \rightarrow \textup{PU}(1,1)$ admits a boundary map $\phi:\bbS^1 \times \Omega \rightarrow \bbS^1$ by \cite[Proposition 3.3]{MonShal0}. In this situation non-elementary means that the algebraic hull of $\sigma$ does not lie inside an elementary subgroup of $\textup{PU}(1,1)$. 

For more general group of Hermitian type the existence of boundary map in the Zariki dense case is not known to the author. Since $\bbS^1 \times \Omega$ is an amenable $\Gamma$-space by \cite[Proposition 4.3.4]{zimmer:libro}, there exists a probability measure-valued map $\widehat{\phi}:\bbS^1 \times \Omega \rightarrow \calM^1(\check{\calS}_{\calX})$. To show that $\widehat{\phi}$ takes values in the set of Dirac measures and hence to obtain a boundary map into the Shilov boundary, one could try to verify that $\sigma$ acts proximally on $\check{\calS}_{\calX}$ (see \cite[Theorem 4.5]{furst:articolo}). For representations proximality boils down to the existence of a $\bbR$-regular element in the image (see \cite[Theorem 3.7]{margulis:libro} and \cite[Proposition 7.2]{BI04} for Hermitian Lie groups). Since the Zariski density implies the existence of a $\bbR$-regular element \cite[Theorem 3.4]{BI04}, a boundary map exists in the case of Zariski dense representations. In the context of Zariski dense measurable cocycle it is more likely that one needs the existence of a measurable family of $\bbR$-regular elements because of the characterization of proximality in terms of slices of the boundary map \cite[Lemma 3.3]{furst:articolo}. Unfortunately at the moment we cannot say something more relevant about it. 

\begin{oss}
In \cite{savini:sarti} the author, together with Sarti, analyze the case of maximal cocycle admitting a boundary map with Zariski dense slices. The latter property clearly is stronger than the Zariski density assumption we made here, but in \cite[Proposition 4.2]{savini:sarti} we proved that if $(\Omega,\mu_\Omega)$ is a $\Gamma$-ergodic space, the Zariski density of the cocycle implies the Zariski density of the slices of the associated boundary map. 
\end{oss}

Before studying more carefully the algebraic hull of $\sigma$, we need to remind briefly some notation regarding the triple products studied by Burger, Iozzi and Wienhard \cite{BIW07}. If we denote by $(\check{S}_{\calX})^{(3)}$ the set of triples of distinct points in $\check{\calS}_{\calX}$, the \emph{Hermitian triple product} is defined as

$$
\langle \langle \cdot , \cdot , \cdot \rangle \rangle: (\check{S}_ {\calX})^{(3)} \rightarrow \bbR^\times \backslash \bbC^\times \ ,
$$
$$
\langle\langle \xi,\eta,\omega \rangle\rangle=e^{i \pi p_{\calX} \beta_{\calX}(\xi,\eta,\omega)} \ \  \text{mod} \bbR^\times \ ,
$$
for every $(\xi,\eta,\omega) \in \check{S}^{(3)}_{\calX}$. The number $p_{\calX}$ is an integer defined in terms of the root system associated to $G$. 

Recall that $\check{\calS}_{\calX}$ is a homogeneous $G$-space, which can be realized as the quotient $G/Q$, where $Q=\mathbf{Q}(\bbR)$ and $\mathbf{Q}$ is a maximal parabolic subgroup of $\mathbf{G}$. Burger, Iozzi and Wienhard were able to extend the Hermitian triple product to a \emph{complex Hermitian triple product} $\langle\langle \cdot, \cdot, \cdot \rangle\rangle_{\bbC}$ defined on $(\mathbf{G}/\mathbf{Q})^3$ with values into $\Delta^\times \backslash A^\times$. Here $A^\times$ is the group $\bbC^\times \times \bbC^\times$ endowed with real structure $(\lambda,\mu) \mapsto (\overline{\mu},\overline{\lambda})$ and $\Delta^\times$ is the image of the diagonal embedding of $\bbC^\times$. More precisely, the authors \cite[Section 2.4]{BIW07} showed that the following diagram commutes
$$
\xymatrix{
(\check{\calS}_{\calX})^{(3)} \ar[rr]^{\langle \langle \cdot, \cdot, \cdot \rangle\rangle} \ar[d]^{(\imath)^3} && \bbR^\times \backslash \bbC^\times \ar[d]^\Delta\\
(\mathbf{G}/\mathbf{Q})^3 \ar[rr]^{\langle \langle \cdot, \cdot, \cdot \rangle\rangle_{\bbC}} && \Delta^\times \backslash A^\times \ .
}
$$
where $\imath:\check{\calS}_{\calX} \rightarrow \mathbf{G}/\mathbf{Q}$ is the map given by the $G$-equivariant identification of $\check{\calS}_{\calX}$ with $(\mathbf{G}/\mathbf{Q})(\bbR)$ and $\Delta$ is the diagonal embedding. 

It is worth mentioning that the complex Hermitian triple product is a rational function on $(\mathbf{G}/\mathbf{Q})^3$ since it can be written as a product of determinants of complex automorphy kernels (see \cite[Equation 2.4]{BIW07}).

Given any pair of transverse points $(\xi,\eta) \in (\check{\calS}_{\calX})^{(2)}$, following \cite[Section 5.1]{BIW07}, we denote by $\calO_{\xi,\eta}$ the open Zariski subset of $\mathbf{G}/\mathbf{Q}$ of points transverse to both $\xi$ and $\eta$. On the set $\calO_{\xi,\eta}$ we have that the map
$$
p_{\xi,\eta}:\calO_{\xi,\eta} \rightarrow \Delta^\times \backslash A^\times, \hspace{5pt} p_{\xi,\eta}(\omega):=\langle \langle \xi, \eta, \omega \rangle \rangle_{\bbC} \ ,
$$
is well-defined and algebraic, by the rationality of the complex Hermitian triple product. Burger, Iozzi and Wienhard \cite[Lemma 5.1]{BIW07} proved that if there exists an integer $m \in \bbZ \setminus \{ 0 \}$ such that $\omega \mapsto p_{\xi,\eta}(\omega)^m$ is constant, then $\calX$ must be of tube type. 

Now we can proceed proving the following theorem, which should be thought of as a generalization of \cite[Theorem 4.1(1)]{BIW1}. 
\begin{teor}\label{teor:symmetric:tube}
Let $L$ be a finite connected covering of $\pu(1,1)$ and let $\Gamma \leq L$ be a torsion-free lattice. Let $(\Omega,\mu_\Omega)$ be a standard Borel probability $\Gamma$-space and let $\sigma:\Gamma \times \Omega \rightarrow G$ be a measurable cocycle which is Zariski dense. Assume that there exists a boundary map $\phi:\bbS^1 \times \Omega \rightarrow \check{\calS}_{\calX}$. If $\sigma$ is maximal, then $\calX$ must be of tube type. 
\end{teor}

\begin{proof}
Consider a positively oriented triple of distinct points $\xi,\eta,\omega \in \bbS^1$. By the maximality assumption we have that $\textup{t}_b(\sigma)=\rk(\calX)$. Recalling that the cocycle $\beta_{\bbS^1}$ is the orientation cocycle rescaled by $\frac{1}{2}$ and substituting the value of the Toledo invariant in Equation \eqref{eq:formula} we obtain
\begin{equation}\label{eq:formula:maximal}
\int_{\Gamma \backslash L}\int_\Omega \beta_{\calX}(\phi(\overline{g}\xi,s),\phi(\overline{g}\eta,s),\phi(\overline{g}\omega,s))d\mu_\Omega(s)d\mu_{\Gamma \backslash L}(\overline{g})=\frac{\rk(\calX)}{2} \ .
\end{equation}
Being $\beta_{\calX}$ bounded from above by $\frac{\rk(\calX)}{2}$, for almost every $\overline{g} \in \Gamma \backslash L$ and almost every $s \in \Omega$ it must hold
$$
\beta_{\calX}(\phi(\overline{g}\xi,s),\phi(\overline{g}\eta,s),\phi(\overline{g}\omega,s))=\frac{\rk(\calX)}{2} \ ,
$$
and by the equivariance of the map $\phi$ it follows
\begin{equation}\label{eq:almost:every:maximal}
\beta_{\calX}(\phi(g\xi,s),\phi(g\eta,s),\phi(g\omega,s))=\frac{\rk(\calX)}{2} \ ,
\end{equation}
for almost every $g \in L$ and almost every $s \in \Omega$. 

For almost every $s \in \Omega$, we know that the $s$-slice $\phi_s:\bbS^1 \rightarrow \check{\calS}_{\calX}, \ \phi_s(\xi):=\phi(\xi,s)$ is measurable by \cite[Lemma 2.6]{fisher:morris:whyte} and, by Equation \ref{eq:almost:every:maximal} it satisfies 
\begin{equation}\label{eq:maximal:slice}
\beta_{\calX}(\phi_s(g\xi),\phi_s(g\eta),\phi_s(g\omega))=\frac{\rk(\calX)}{2} \ ,
\end{equation}
for almost every $g \in L$. Since the same reasoning applies to a negatively oriented triple, we must have
\begin{equation}\label{eq:maximal:slice:triples}
\beta_{\calX}(\phi_s(\xi),\phi_s(\eta),\phi_s(\omega))=\pm \frac{\rk(\calX)}{2} \ ,
\end{equation}
for almost every triple $\xi,\eta,\omega$ such that $\beta_{\bbS^1}(\xi,\eta,\omega)=\pm 1/2$. Recalling that it holds
$$
\langle \langle \phi_s(\xi), \phi_s(\eta), \phi_s(\omega) \rangle \rangle=e^{i\pi p_{\calX} \beta_{\calX}(\xi,\eta,\omega)} \ \textup{mod} \bbR^\times \ ,
$$
Equation \eqref{eq:maximal:slice:triples} implies that
\begin{equation}\label{eq:hermitian:product}
\langle \langle \phi_s(\xi),\phi_s(\eta),\phi_s(\omega) \rangle\rangle^2 =e^{\pm i\pi p_{\calX} \rk(\calX)}= 1 \ \  \textup{mod} \bbR^\times \ ,
\end{equation}
for almost every $\xi,\eta,\omega \in \bbS^1$ distinct. 

 Fix now a pair $(\xi,\eta) \in (\bbS^1)^{2}$ such that Equation \eqref{eq:hermitian:product} holds for almost every $\omega \in \bbS^1$. A particular consequence of maximality is that $\phi_s(\xi)$ and $\phi_s(\eta)$ are transverse for almost every $s \in \Omega$. Denoting by $(\check{\calS}_{\calX})^{(2)}$ the subset of $\check{\calS}_{\calX}^2$ of pairs of transverse points, we have a map 
$$
\Omega \rightarrow (\check{\calS}_{\calX})^{(2)} \ , \ \  s \mapsto (\phi_s(\xi),\phi_s(\eta)) \ ,
$$
which is measurable by the measurability of $\phi$. By the transitivity of $G$ on pairs of transverse points $(\check{\calS}_{\calX})^{(2)}$, the latter can be thought of as the quotient of $G$ by the stabilizer $\textup{Stab}_G((\xi_0,\eta_0))$ 
of a fixed pair $(\xi_0,\eta_0)$. Hence we have a measurable map 
$$
\Omega \rightarrow G/\textup{Stab}_G((\xi_0,\eta_0)) \ ,
$$
and by composing with a measurable section $G/\textup{Stab}_G((\xi_0,\eta_0)) \rightarrow G$ given by \cite[Corollary A.8]{zimmer:libro}, we get a measurable function $f:\Omega \rightarrow G$ such that
$$
\phi_s(\xi)=f(s)\xi_0 \ , \hspace{10pt} \phi_s(\eta)=f(s)\eta_0 \ , \ (\xi_0,\eta_0,f(s)^{-1}\phi_s(\omega)) \ \text{is maximal} 
$$
for almost every $\omega \in \bbS^1, s \in \Omega$. For such a measurable function $f$, we consider $\sigma^f$ and the map $\phi^f$ as the ones defined in Section \ref{sec:measurable:cocycles}. For the ease of notation we are going to write $\alpha=\sigma^f$ and $\psi=\phi^f$. By the choice of the map $f$, Equation \eqref{eq:hermitian:product} can be rewritten as
$$
\langle \langle \xi_0,\eta_0,\psi_s(\omega) \rangle\rangle^2 = 1 \ \  \textup{mod} \bbR^\times \ ,
$$
for almost every $\omega \in \bbS^1, s \in \Omega$. The previous equation implies that $\psi_s(\omega) \in \calO_{\xi_0,\eta_0}$ for almost every $\omega \in \bbS^1$ and almost every $s \in \Omega$. We denote by $E$ the subset of full measure in $\bbS^1 \times \Omega$ such that $\psi_s(\omega) \in \calO_{\xi_0,\eta_0}$ for all $E$. Define 
$$
E^\Gamma:=\bigcap_{\gamma \in \Gamma} \gamma E \ ,
$$
which has full measure being a countable intersection of full measure sets (notice that $\Gamma$ preserves the measure class on $\bbS^1 \times \Omega$). Since $\sigma$ is Zariski dense, the cocycle $\alpha$ is Zariski dense too. Since the Zariski closure of $\psi(E^\Gamma)$ is preserved by the algebraic hull of $\alpha$ which coincides with $\mathbf{G}$, the set $\psi(E^\Gamma)$ is Zariski dense in $\mathbf{G}/\mathbf{Q}$, whence is $\psi(E^\Gamma)$ Zariski dense in $\calO_{\xi_0,\eta_0}$. Thus the map $\omega \rightarrow p_{\xi_0,\eta_0}(\omega)^2$ is constant on $\calO_{\xi_0,\eta_0}$ and $\calX$ is of tube type, as claimed.  
\end{proof}

An important consequence of the previous theorem is the following

\begin{cor}
Let $L$ be a finite connected covering of $\pu(1,1)$ and let $\Gamma \leq L$ be a torsion-free lattice. Let $(\Omega,\mu_\Omega)$ be a standard Borel probability $\Gamma$-space. Consider a maximal measurable cocycle $\sigma:\Gamma \times \Omega \rightarrow G$ and assume that there exists a boundary map $\phi:\bbS^1 \times \Omega \rightarrow \check{\calS}_{\calX}$. If $G$ is not of tube type, then $\sigma$ cannot be Zariski dense.
\end{cor}

As a consequence of Theorem \ref{teor:symmetric:tube}, if $\mathbf{H}$ is the algebraic hull of a maximal cocycle $\sigma$ and $H=\mathbf{H}(\bbR)$, then $H^\circ$ must be a Hermitian group of tube type. 

The following theorem collects all the properties we discovered about the algebraic hull of a maximal cocycle and it should be thought of as a statement equivalent to \cite[Theorem 5]{BIW1} in the context of measurable cocycles. 

\begin{repteor}{teor:maximal:alghull}
Let $\Gamma \leq L$ be a torsion-free lattice and let $(\Omega,\mu_\Omega)$ be a standard Borel probability $\Gamma$-space. Let $\mathbf{G}$ be a semisimple algebraic $\bbR$-group such that $G=\mathbf{G}(\bbR)^\circ$ is a Lie group of Hermitian type. Consider a measurable cocycle $\sigma:\Gamma \times \Omega \rightarrow G$. Denote by $\mathbf{H}$ the algebraic hull of $\sigma$ in $\mathbf{G}$ and set $H=\mathbf{H}(\bbR)^\circ$. If $\sigma$ is maximal, then 
\begin{enumerate}
	\item the algebraic hull $\mathbf{H}$ is reductive;
	\item the centralizer $Z_G(H)$ is compact;
\end{enumerate}
If additionally $\sigma$ admits a boundary map $\phi:\bbS^1 \times \Omega \rightarrow \check{\calS}_{\calY}$ into the Shilov boundary of the symmetric space $\calY$ associated to $H$, then 
\begin{enumerate}
	\item[(3)] the symmetric space $\calY$ is Hermitian of tube type;
	\item[(4)] it holds $\mathbf{H}(\bbR) \subset \textup{Isom}(\calT)$ for some maximal tube-type subdomain $\calT$ of $\calX$. Equivalently there exists a cocycle cohomologous to $\sigma$ which preserves $\calT$. 
\end{enumerate}
\end{repteor}

\begin{proof}
Being maximal, the cocycle $\sigma$ is tight by Proposition \ref{prop:maximal:tight}. Hence we can apply Theorem \ref{teor:alg:hull:tight} to get properties $1)$ and $2)$. Since we assumed the existence of a boundary map, by Theorem \ref{teor:symmetric:tube} the symmetric space $\calY$ must be of tube type, whence point $3)$. 

The inclusion $i: H \rightarrow G$ is tight because the cocycle $\sigma$ is tight. Since the symmetric space $\calY$ associated to $H$ is of tube type and the inclusion is tight, by \cite[Theorem 9(1)]{BIW09} there exists a unique maximal tube type subdomain $\calT$ of $\calX$ preserved by $H$. By uniqueness, $\calT$ must be preserved by the whole $\mathbf{H}(\bbR)$ and we are done.
\end{proof}

In the particular case that $G=\textup{PU}(n,1)$, then maximality of a measurable cocycle $\sigma:\Gamma \times \Omega \rightarrow \textup{PU}(n,1)$ is sufficient to recover that fact that $\sigma$ is cohomologous to a measurable cocycle which preserves a complex geodesic (in the case of representations see \cite[Theorem 8]{BIgeo},\cite[Theorem C]{koziarz:maubon}), as shown by the following

\begin{repprop}{prop:complex:geodesic}
Let $L$ be a finite connected covering of $\textup{PU}(1,1)$ and let $\Gamma \leq L$ be a torsion-free lattice. Let $(\Omega,\mu_\Omega)$ a standard Borel probability $\Gamma$-space. If a measurable cocycle $\sigma:\Gamma \times \Omega \rightarrow \textup{PU}(n,1)$ is maximal, then it is cohomologous to a measurable cocycle which preserves a complex geodesic. 
\end{repprop}

\begin{proof}
The proof is very similar to the one of Theorem \ref{teor:symmetric:tube}, so here we are going to omit some details. 

Since $\sigma:\Gamma \times \Omega \rightarrow \textup{PU}(n,1)$ is maximal, it cannot be elementary, otherwise its Toledo invariant would vanish. Thus we can apply \cite[Proposition 3.3]{MonShal0} to a get boundary map $\phi:\bbS^1 \times \Omega \rightarrow \partial_\infty \mathbb{H}^n_{\bbC}$.

Recall that in the particular case of $\textup{PU}(n,1)$ the Bergmann cocycle is nothing else than the Cartan cocycle rescaled by $\frac{1}{2}$. Choose now a triple of points $\xi,\eta,\omega \in \bbS^1$ with sign $\pm 1$. Using Equation \eqref{eq:formula}, the maximality assumption implies that 
$$
\int_{\Gamma \backslash L} \int_\Omega c_n(\phi(\overline{g}\xi,s),\phi(\overline{g}\eta,s),\phi(\overline{g}\omega,s))d\mu_\Omega(s)d\mu_{\Gamma \backslash L}(\overline{g})= \pm 1 \ ,
$$ 
where $c_n$ is the Cartan cocycle. As a consequence we have that 
\begin{equation}\label{eq:cartan:chain}
c_n(\phi_s(g\xi),\phi_s(g\eta),\phi_s(g\omega))=\pm 1 \ ,
\end{equation}
for almost every $s \in \Omega,\ g \in L$ and the sign is the same of the one of the triple $(\xi,\eta,\omega)$.

Fix now $\xi,\eta \in \bbS^1$ distinct. Recall that a \emph{chain} is the boundary of a complex geodesic and given two distinct points in $\partial_\infty \bbH^n_{\bbC}$ there is a unique chain passing through them. We denote by $C_s$ the chain passing through $\phi_s(\xi)$ and $\phi_s(\eta)$. Notice that $C_s$ depends measurably on $s \in \Omega$ because $\phi_s$ varies measurably with respect to it by the measurability of $\phi$. 

Since the Cartan invariant is equal to $\pm 1$ if and only if the distinct points in the triple lie on the same chain \cite[Lemma 6.3]{BIgeo}, by Equation \eqref{eq:cartan:chain} it holds
$$
\phi_s(\omega) \in C_s \ ,
$$
for almost every $s \in \Omega$ and almost every $\omega \in \bbS^1$. In this way we obtain that 
$$
\textup{EssIm}(\phi_s) \subset C_s \ , \ \ \ C_{\gamma s}=\sigma(\gamma,s)C_s \ , 
$$
where the second equation follows by the $\sigma$-equivariance of $\phi$. Choose now a chain $C_0$ in $\partial_\infty \mathbb{H}^n_{\bbC}$. The transitivity of the action of $\textup{PU}(n,1)$ on the space of chains, implies the existence of a measurable map $f:\Omega \rightarrow \textup{PU}(n,1)$ such that 
$$
f(s)(C_0)=C_s \ ,
$$
by a similar argument to the one in the proof of Theorem \ref{teor:symmetric:tube}. Thus the twisted cocycle $\sigma^f$ preserves $C_0$, indeed we have
$$
C_0=f(\gamma s)^{-1} C_{\gamma s}=f(\gamma s)^{-1}\sigma(\gamma,s)C_s=f(\gamma s)^{-1}\sigma(\gamma,s)f(s)C_0=\sigma^f(\gamma,s)C_0 \ ,
$$
and the claim is proved.
\end{proof}

\subsection{Regularity properties of boundary maps}\label{sec:boundary:map}

Imitating what happens in the context of representations, we are going to study the regularity properties of boundaries map associated to maximal measurable cocycles. Given a maximal Zariski dense measurable cocycle, under suitable hypothesis on the push-forward measure with respect to the slices of the boundary map, we are going to show that there exists an essentially unique equivariant measurable map with left-continuous (respectively right-continuous) slices which preserve transversality and maximality. We are going to follow the line of \cite[Section 5]{BIW1}. 

Before introducing the setup of the section, we say that a measurable map $\phi:\bbS^1 \rightarrow \check{\calS}_{\calX}$ is \emph{maximal} if it satisfies Equation \eqref{eq:maximal:slice}. Notice that almost every slice of a boundary map associated to a maximal cocycle is maximal. If a measurable map $\phi$ is maximal, we will similarly say that its essential graph $\textup{EssGr}(\phi)$ is \emph{maximal}. Recall that the essential graph is the support of the push-forward of the Lebesgue measure on $\bbS^1$ with respect to the map $\xi \mapsto (\xi,\phi(\xi)) \in \bbS^1 \times \check{\calS}_{\calX}$. 

\begin{setup}\label{setup:boundary:map}
From now until the end of the section we are going to assume the following 
\begin{itemize}
\item $\Gamma \leq L$ be a torsion-free lattice of a finite connected covering $L$ of $\pu(1,1)$;
\item $(\Omega,\mu_\Omega)$ is a standard Borel probability $\Gamma$-space;
\item $\sigma:\Gamma \times \Omega \rightarrow G$ is a maximal Zariski dense cocycle with boundary map $\phi:\bbS^1 \times \Omega \rightarrow \check{\calS}_{\calX}$;
\item denote by $\{E_s\}_{s \in \Omega}$ the family of essential graphs $E_s=\textup{EssGr}(\phi_s)$ associated to the slices.
\end{itemize}
\end{setup}

Having introduced the setup we needed, we can now move on proving the following 

\begin{lem}\label{lemma:maximality:triples}
In the situation of Setup \ref{setup:boundary:map}, suppose that $E_s$ is maximal. Let $(\xi_i,\eta_i) \in E_s$ for $i=1,2,3$ be points such that $\xi_1,\xi_2,\xi_3$ are pairwise distinct and $\eta_1,\eta_2,\eta_3$ are pairwise transverse. Then it holds
$$
\beta_{\calX}(\eta_1,\eta_2,\eta_3)=\rk(\calX) \beta_{\bbS^1}(\xi_1,\xi_2,\xi_3) \ . 
$$
\end{lem}

\begin{proof}
Denote by $I_i$ for $i=1,2,3$ open paiwise non-intersecting intervals such that $\xi_i \in I_i$ and for any $\omega_i \in I_i$ it holds
$$
\beta_{\bbS^1}(\omega_1,\omega_2,\omega_3)=\beta_{\bbS^1}(\xi_1,\xi_2,\xi_3) \ . 
$$
Consider a open neighborhood $U_i$ of $\eta_i$, for $i=1,2,3$, such that $U_1 \times U_2 \times U_3 \in (\check{\calS}_{\calX})^{(3)}$. Then the measurable set 
$$
A_i=\{ \omega \in I_i \ | \ \phi_s(\omega_i) \in U_i \} \ ,
$$
is a set of positive measure, since $\eta_1,\eta_2,\eta_3$ are in the essential image of $\phi_s$. Since we assumed the slice $E_s$ is maximal, for almost every $(\omega_1,\omega_2,\omega_3) \in A_1 \times A_2 \times A_3$ we have that 
$$
\beta_{\calX}(\phi_s(\omega_1),\phi_s(\omega_2),\phi_s(\omega_3))=\rk(\calX)\beta_{\bbS^1}(\omega_1,\omega_2,\omega_3)=\rk(\calX)\beta_{\bbS^1}(\xi_1,\xi_2,\xi_3) \ .
$$

By setting $\varepsilon=2\beta_{\bbS^1}(\xi_1,\xi_2,\xi_3)$, we have that $|\varepsilon|=1$ and for almost every $(\omega_1,\omega_2,\omega_3) \in A_1 \times A_2 \times A_3$ we have that 
$$
(\phi_s(\omega_1),\phi_s(\omega_2),\phi_s(\omega_3)) \in U_1 \times U_2 \times U_3 \cap \calO_{\varepsilon \rk{\calX}} \ ,
$$
where $\calO_{\varepsilon \rk{\calX}}$ is the open set in $\check{\calS}_{\calX}^3$ on which $\beta_{\calX}$ is identically equal to $\rk(\calX)/2$. By the arbitrary choice of the neighborhood $U_i$, must have $(\eta_1,\eta_2,\eta_3) \in \overline{\calO_{\varepsilon \rk \calX}}$. 

Since we have that 
$$
\overline{\calO_{\varepsilon \rk \calX}} \cap (\check{\calS}_{\calX})^{(3)}=\overline{\calO_{\varepsilon \rk \calX}} \cap ( \sqcup_{i=0}^{\rk(\calX)}  \calO_{-\rk \calX + 2 i}) = \calO_{\varepsilon \rk \calX}
$$
and $(\eta_1,\eta_2,\eta_3) \in (\check{\calS}_{\calX})^{(3)}$, the triple is maximal and the claim follows. 
\end{proof}

In order to proceed we have now to discuss a condition we have to impose on the slices of the boundary map. Recall that $\check{\calS}_{\calX}$ can be identified with $G/Q$, where $Q$ is a maximal parabolic subgroup. We denote by $\textbf{V}_\xi \subset \mathbf{G}/\mathbf{Q}$ the Zariski closed set of points transverse to $\xi$ and set $V_\xi:=\textbf{V}_\xi(\bbR)$, the set of points transverse to $\xi$ in the Shilov boundary. 

Burger, Iozzi and Wienhard \cite[Proposition 5.2]{BIW1} proved that the boundary map associated to a Zariski dense representation has very strong properties, since its essential image intersects any proper Zariski closed set of the Shilov boundary in a set of measure zero. The author wonders under which hypothesis the same property should hold for almost every slice of a boundary map associated to a cocycle. Here we are going to assume it. More precisely

\begin{assumption}\label{ass:zariski:zero:measure}
In the situation of Setup \ref{setup:boundary:map}, we suppose that for every proper Zariski closed set $\mathbf{V} \subset \mathbf{G}/\mathbf{Q}$ it holds 
$$
\nu(\phi_s^{-1}(\mathbf{V}(\bbR)))=0 \ ,
$$
for almost every $s \in \Omega$. Here $\nu$ is the round measure on $\bbS^1$. 
\end{assumption}

Assumption \ref{ass:zariski:zero:measure} is satisfied in the trivial case $G=\textup{PU}(1,1)$ but also by cocycles which are cohomologous to a Zariski dense representation $\rho:\Gamma \rightarrow G$, as a consequence of \cite[Proposition 5.2]{BIW1}. We are not aware if this property can be extended to a wider class of cocycles. 

\begin{lem}\label{lemma:transverse}
Let $E_s$ be a maximal graph satisfying Assumption \ref{ass:zariski:zero:measure} and let $(\xi_1,\eta_1), (\xi_2,\eta_2) \in E_s$ with $\xi_1 \neq \xi_2$. Then $\eta_1$ and $\eta_2$ are transverse. 
\end{lem}

\begin{proof}
For any distinct $\xi,\omega \in \bbS^1$ we denote by 
$$
((\xi,\omega)):=\{ \eta \in \bbS^1 \ | \ \beta_{\bbS^1}(\xi,\zeta,\omega)=\frac{1}{2} \} \ .
$$
Thanks to Assumption \ref{ass:zariski:zero:measure}, we can suppose that the essential image of the slice $\phi_s$ meets any Zariski closed set in a measure zero set. Hence we can find $\alpha_1 \in ((\xi_1,\xi_2))$ such that $\phi_s(\alpha_1)$ is transverse to both $\eta_1$ and $\eta_2$. In the same way there will exist a point $\alpha_2 \in ((\xi_2,\xi_1))$ such that $\phi_s(\alpha_2)$ is transverse to $\eta_1$ and $\eta_2$. 

Using now jointly Lemma \ref{lemma:maximality:triples} and the cocycle condition on $\beta_{\calX}$ we get
\begin{align*}
0&=\beta_{\calX}(\phi_s(\alpha_1),\eta_1,\phi_s(\alpha_2))-\beta_{\calX}(\eta_1,\eta_2,\phi_s(\alpha_2))+\\
 &+\beta_{\calX}(\eta_1,\phi_s(\alpha_1),\phi_s(\alpha_2))-\beta_{\calX}(\eta_1,\phi_s(\alpha_1),\eta_2))=\\
 &=\frac{\rk(\calX)}{2}-\beta_{\calX}(\eta_1,\eta_2,\phi_s(\alpha_2))+\frac{\rk(\calX)}{2}-\beta_{\calX}(\eta_1,\phi_s(\alpha_1),\eta_2) \ . \\
\end{align*}
The previous line implies that $\beta_{\calX}(\eta_1,\eta_2,\phi_s(\alpha_2))=\frac{\rk(\calX)}{2}$ and hence $\eta_1$ and $\eta_2$ are transverse. 
\end{proof}

Given now any subset $A \subset \bbS^1$ we put
$$
F_A^s:=\{ \eta \in \check{\calS}_{\calX} \  | \ \exists \  \xi \in A \ : \ (\xi,\eta) \in E_s \} \ . 
$$
We define also $$((\xi,\omega]]:=((\xi,\omega)) \cup \{ \omega \} \ . $$

\begin{lem}\label{lemma:one:point}
Let $s \in \Omega$ be a point such that $E_s$ is a maximal graph satisfying Assumption \ref{ass:zariski:zero:measure}. Let $\xi \neq \omega$ be two points in $\bbS^1$. Then $\overline{F^s_{((\xi,\omega]]}} \cap F^s_\xi$ and $\overline{F^s_{[[\omega,\xi))}} \cap F^s_{\xi}$ consist each of one point.  
\end{lem}

\begin{proof}
We prove that $\overline{F^s_{((\xi,\omega]]}} \cap F^s_\xi$ consists of exactly one point. The same strategy can be applied to $\overline{F^s_{[[\omega,\xi))}} \cap F^s_{\xi}$  to prove the same statement. 

Let $\eta,\eta' \in \overline{F^s_{((\xi,\omega]]}} \cap F^s_\xi$ and consider $(\xi_n,\eta_n) \in E_s$ a sequence such that 
$$
\xi_n \in ((\xi,\omega]], \ \ \ \lim_{n \to \infty} \xi_n=\xi, \ \ \ \lim_{n \to \infty} \eta_n=\eta \ .
$$

Given any $\zeta \in ((\xi,\omega))$, we can apply the same reasoning of \cite[Lemma 5.8]{BIW1}, to say that 
$$
\overline{F^s_{((\xi,\omega]]}} \cap F^s_\xi = \overline{F^s_{((\xi,\zeta]]}} \cap F^s_\xi \ .
$$
Thanks to the previous equation, consider a sequence $(\omega_n,\eta'_n) \in E_s$ so that 
$$
\omega_n \in ((\xi,\xi_n)), \ \ \  \lim_{n \to \infty} \omega_n=\xi, \ \ \ \lim_{n \to \infty} \eta'_n=\eta' \ .
$$

Applying Lemma \ref{lemma:transverse} we have that $\eta,\eta'_n,\eta_n$ are pairwise transverse. Hence we can apply Lemma \ref{lemma:maximality:triples} to the triples $(\xi,\omega_n,\xi_n)$ and $(\eta,\eta_n',\eta_n)$ to get
$$
\beta_{\calX}(\eta,\eta_n',\eta_n)=\rk(\calX)\beta_{\bbS^1}(\xi,\omega_n,\xi_n)=\frac{\rk(\calX)}{2} \ .
$$
Since $\lim_{n \to \infty} \eta_n=\eta$, Property $4)$ of Section \ref{sec:hermitian:groups} of the Bergmann cocycles $\beta_{\calX}$ forces $\lim_{n \to \infty} \eta'_n=\eta$ and hence $\eta=\eta'$. 
\end{proof}

In this way wet get immediately the following 

\begin{cor}\label{cor:two:points}
Let $s \in \Omega$ be a point such that $E_s$ is a maximal graph satisfying Assumption \ref{ass:zariski:zero:measure}. For every $\xi \in \bbS^1$ the set $F^s_\xi$ contains either one or two points. 
\end{cor}

\begin{proof}
Consider $\omega_-, \xi, \omega_+ \in \bbS^1$ and let $\eta \in F^s_\xi$. Since it holds
$$
F^s_\xi=\left( \overline{F^s_{[[\omega_-,\xi))}} \cap F^s_{\xi} \right) \cup \left( \overline{F^s_{((\xi,\omega_+]]}} \cap F^s_\xi \right)\ ,
$$
the claim follows by Lemma \ref{lemma:one:point}.
\end{proof}

We are know ready to prove the main theorem of the section which extends in some sense \cite[Theorem 5.1]{BIW1} to the context of measurable cocycles. 

\begin{teor}\label{teor:boundary:map}
In the situation of Assumption \ref{ass:zariski:zero:measure}, there exist two measurable maps
$$
\phi^\pm:\bbS^1 \times \Omega \rightarrow \check{\calS}_{\calX} 
$$
such that
\begin{enumerate}
\item the slice $\phi^+_s:\bbS^1 \rightarrow \check{\calS}_{\calX}$ is right continuous for almost every $s \in \Omega$;
\item the slice $\phi^-_s:\bbS^1 \rightarrow \check{\calS}_{\calX}$ is left continuous for almost every $s \in \Omega$;
\item the maps $\phi^\pm$ are measurable and $\sigma$-equivariant;
\item for every $\xi \neq \omega $ in $\bbS^1$ and almost every $s \in \Omega$, $\phi^\varepsilon_s(\xi)$ is transverse to $\phi^\delta_s(\omega)$, where $\varepsilon, \delta \in \{  \pm \}$;
\item almost every slice is monotone, that is for every $\xi,\omega,\zeta \in \bbS^1$ and almost every $s \in \Omega$ it holds
$$
\beta_{\calX}(\phi_s^\varepsilon(\xi),\phi_s^\delta(\omega),\phi_s^\theta(\zeta))=\rk(\calX) \beta_{\bbS^1}(\xi,\omega,\zeta) \ ,
$$
where $\varepsilon,\delta,\theta \in \{ \pm \}$. 
\end{enumerate}
\end{teor}

\begin{proof}
By assumption we know that for almost every $s \in \Omega$, the slice $\phi_s$ is maximal and it satisfies Assumption \ref{ass:zariski:zero:measure}. For any such $s$, we define for every $\xi \in \bbS^1$ the following maps
$$
\phi_s^+(\xi)=\overline{F^s_{[[\omega_-,\xi))}} \cap F^s_{\xi}\ , 
\phi_s^-(\xi)=\overline{F^s_{((\xi,\omega_+]]}} \cap F^s_\xi  \ ,
$$
where $\omega_- ,\xi,\omega_+$ is a positively oriented triple in $\bbS^1$ and $\omega_\pm$ are arbitrary. The right continuity of $\phi^+_s$ and the left continuity of $\phi^-_s$ are clear by their definitions. We can define
$$
\phi^\pm:\bbS^1 \times \Omega \rightarrow \check{\calS}_{\calX}, \ \phi_s^\pm(\xi,s):=\phi_s^\pm(\xi) \ .
$$
The measurability of the functions $\phi^\pm_s$ comes from the fact the slice $\phi_s^\pm$ are measurable and varies measurably with respect to $s$ by the measurability of $\phi$. The $\sigma$-equivariance of the latter implies that $\phi^\pm$ are $\sigma$-equivariant. 

Finally property $4)$ follows by Lemma \ref{lemma:transverse} and property $5)$ follows by Lemma \ref{lemma:maximality:triples}. 
\end{proof}

\bibliographystyle{amsalpha}
\bibliography{biblionote}
\end{document}